\documentclass[letterpaper,leqno,11pt,oneside]{amsart}

\usepackage{amsmath,amsthm,amsfonts,amssymb}
\usepackage{enumitem}
\usepackage[usenames]{color}
\usepackage{mathtools,comment}
\usepackage[extension=pdf]{hyperref}
\usepackage[height=9.6in,width=5.95in]{geometry}
\usepackage{graphics,graphicx}
\usepackage[final,color,notref,notcite]{showkeys}
\usepackage{xr}

\usepackage[style=alphabetic,natbib=true,maxnames=99,isbn=false,doi=false,url=false,firstinits=true,hyperref=auto,arxiv=abs]{biblatex}
\DeclareFieldFormat[article,inbook,incollection,inproceedings,patent,thesis,unpublished]{title}{#1\isdot}
\renewbibmacro{in:}{%
  \ifentrytype{article}{}{%
  \printtext{\bibstring{in}\intitlepunct}}}
\defbibheading{apa}[\refname]{\section*{#1}}

\definecolor{darkblue}{rgb}{0.13,0.13,0.39}%{0.03,0.03,0.265}
\hypersetup{colorlinks=true,urlcolor=darkblue,citecolor=darkblue,linkcolor=darkblue,%
  pdftitle=Local behavior and hitting probabilities of Airy1}

\mathtoolsset{showonlyrefs,showmanualtags}

\newtheorem{thm}{Theorem}
\newtheorem{lem}{Lemma}[section]
\newtheorem{prop}[lem]{Proposition}

\theoremstyle{definition}
\newtheorem{rem}[lem]{Remark}
\newtheorem*{rem*}{Remark}

\newcounter{assum}

\newcommand{\I}{{\rm i}}
\newcommand{\pp}{\mathbb{P}}

\newcommand{\ee}{\mathbb{E}}
\newcommand{\rr}{\mathbb{R}}
\newcommand{\nn}{\mathbb{N}}

\newcommand{\aip}{\mathcal{A}_2}
\newcommand{\aipo}{\mathcal{A}_1}

\newcommand{\p}{\partial}
\newcommand{\uno}[1]{\mathbf{1}_{#1}}

\newcommand{\ep}{\varepsilon}
\newcommand{\vs}{\vspace{6pt}}
\newcommand{\wt}{\widetilde}

\newcommand{\K}{K_{\Ai}}

\newcommand{\sT}{\mathsf{T}}
\newcommand{\sK}{\mathsf{K}}
\newcommand{\sU}{\mathsf{U}}
\newcommand{\sV}{\mathsf{V}}
\newcommand{\tsWo}{\mathsf{W_3}}
\newcommand{\tsWt}{\mathsf{W_4}}
\newcommand{\sWo}{\mathsf{W_1}}
\newcommand{\sWt}{\mathsf{W_2}}
\newcommand{\sP}{\mathsf{P}}
\newcommand{\sI}{\mathsf{I}}

\DeclareMathOperator{\Ai}{Ai}
\DeclareMathOperator{\tr}{tr}

\newcommand{\gref}[1]{\ref*{g-#1} of \cite{cqr}}
\newcommand{\geqref}[1]{(\ref*{g-#1}) in \cite{cqr}}

\numberwithin{equation}{section}

\let\oldmarginpar\marginpar
\renewcommand\marginpar[1]{\-\oldmarginpar[\raggedleft\footnotesize #1]%
{\raggedright{\small\textsf{#1}}}}

\allowdisplaybreaks[2]

\begin{document}

\title[Local behavior and hitting probabilities of \texorpdfstring{Airy$_1$}{Airy1}]{%
  Local behavior and hitting probabilities of the \texorpdfstring{Airy$_1$}{Airy1} process}

\author{Jeremy Quastel}
\address[J.~Quastel]{
  Department of Mathematics\\
  University of Toronto\\
  40 St. George Street\\
  Toronto, Ontario\\
  Canada M5S 2E4} \email{quastel@math.toronto.edu}
\author{Daniel Remenik}
\address[D.~Remenik]{
  Department of Mathematics\\
  University of Toronto\\
  40 St. George Street\\
  Toronto, Ontario\\
  Canada M5S 2E4 \newline \indent\textup{and}\indent
  Departamento de Ingenier\'ia Matem\'atica\\
  Universidad de Chile\\
  Av. Blanco Encala\-da 2120\\
  Santiago\\
  Chile} \email{dremenik@math.toronto.edu}

\maketitle

\begin{abstract}
  We obtain a formula for the $n$-dimensional distributions of the Airy$_1$ process in
  terms of a Fredholm determinant on $L^2(\rr)$, as opposed to the standard formula which
  involves extended kernels, on $L^2(\{1,\dots,n\}\times\rr)$. The formula is analogous to
  an earlier formula of \citet{prahoferSpohn} for the Airy$_2$ process. Using this formula
  we are able to prove that the Airy$_1$ process is H\"older continuous with exponent
  $\frac12-$ and that it fluctuates locally like a Brownian motion. We also explain how
  the same methods can be used to obtain the analogous results for the Airy$_2$ process. As
  a consequence of these two results, we derive a formula for the continuum statistics of
  the Airy$_1$ process, analogous to that obtained in \cite{cqr} for the Airy$_2$ process.
\end{abstract}

\section{Introduction and Main Results}

\subsection{General background}

The Airy processes are stochastic processes which are expected to govern the asymptotic
spatial fluctuations in a wide variety of random growth models on a one dimensional
substrate, top lines of non-intersecting random walks and free energies of directed random
polymers in 1 + 1 dimensions (all belonging to the Kardar-Parisi-Zhang, or KPZ,
universality class \cite{kpz}). They are non-Markovian and are defined in terms of their
finite-dimensional distributions, which are given by determinantal formulas.  These
formulas, which have been derived by asymptotic analysis of exact formulas in special
discrete models such as the totally asymmetric simple exclusion process and the
polynuclear growth model, give the $n$-dimensional distributions in terms of Fredholm
 determinants of extended kernels, on $L^2(\{1,\dots,n\}\times\rr)$.  The exact results are
then conjecturally extrapolated to more general processes in the universality class which
do not possess the same exact solvability.

The particular Airy process arising in each case depends on the initial data, and this
picks out a number of KPZ sub-universality classes.  For reasons of scaling invariance,
there are three special pure initial data classes: {\it narrow wedge, flat, and
  equilibrium.}  {\it Narrow wedge} corresponds to point-to-point polymers, or growth
models where the exponential of the height is initially a Dirac delta.  Physically, one
starts with curved, or droplet, initial data. After some time $t$, the height looks like a
parabola in space, corresponding to the deterministic evolution, on top of which is
approximately an Airy$_2$ process \cite{prahoferSpohn} with amplitude $t^{1/3}$ and
varying on a spatial scale of $t^{2/3}$.  {\it Flat} corresponds to point-to-line
polymers, or growth models with constant initial data. At time $t$, one sees spatially the
Airy$_1$ process \cite{sasamoto}, again with size $t^{1/3}$ and varying on spatial scale
$t^{2/3}$.  {\it Equilibrium} corresponds to growth models starting from equilibrium,
which in the KPZ universality class means approximately a two-sided Brownian motion.  At a
later time one sees spatially the Airy$_{{\rm stat}}$ process
\cite{baikFerrariPeche}. Note that all these descriptions are modulo a global height shift
which is non-trivial itself, and can be very large compared to the scales on which these
fluctuations are observed.

There are also three other basic mixed initial data, corresponding to starting with one of
the basic three geometries to the left of the origin and another one to the right.  The
resulting spatial fluctations are still of size $t^{1/3}$ and on a spatial scale of
$t^{2/3}$, with non-homogeneous crossover Airy processes Airy$_{2\to1}$ \cite{bfs},
Airy$_{1\to{\rm stat}}$ \cite{bfsTwoSpeed} and Airy$_{2\to{\rm stat}}$
\cite{sasImamPolyHalf,corwinFerrariPeche}, the names being self-explanatory.  Of course,
there will be other less commonly seen sub-universality classes, but these six are the
basic ones, and, interestingly, all have determinantal
finite-dimensional distributions.

Although the determinantal formulas arise naturally in deriving the finite-dimensional
distributions from the special solvable discrete models, they are cumbersome for the
analysis of properties of these processes involving short range scales. For example, one
would expect to be able prove the pathwise continuity directly by just checking the
Kolmogorov continuity criterion using the determinantal formula for the two point
distributions with extended kernel on $L^2(\{1,2\}\times\rr)$.  This turned out to be
surprisingly difficult, and has been an open problem since the processes were
introduced. For the Airy$_2$ process, which is in some sense the most basic one, what was
done historically was to study the probability measure on the point processes obtained by
sampling the Airy line ensemble at a finite set of times.  \citet{prahoferSpohn} proved
the continuity of the Airy line ensemble as a point process, from which the continuity of
the top line, the Airy$_2$ process, would follow if one knew that the points came from a
non-intersecting line ensemble. However, this was not known at the time (though it is now,
see \cite{corwinHammond}). \citet{Johansson} proved the tightness of an approximating line
ensemble (the multilayer PNG model), which in particular implied the continuity of the
Airy$_2$ process.

On the other hand, the other processes do not arise easily as top lines of line ensembles.  For example,
for the Airy$_1$ process, which will be our main example in this article,  
 even the continuity remained open. 
 
 One also hopes to study variational problems involving the Airy processes.  These arise
 naturally.  A well-known example is the famous result of \citet{Johansson} that the
 supremum of the Airy$_2$ process minus a parabola has the Tracy-Widom GOE distribution
 \cite{tracyWidom2}. There is also a generalization of this \cite{qr-airy1to2} that the
 same supremum on a half-line is given by the one point marginal of the Airy$_{2\to1}$
 process. Variational problems naturally involve infinitely many spatial points, so
 formulas giving the distribution of $n$ sample points in terms of determinants of
 extended kernels on $L^2(\{1,\dots,n\}\times\rr)$ are not a good tool.  In \cite{cqr} we
 introduced a continuum formula for the Airy$_2$ process, which gives the probability that
 the process lies below a given function on an arbitrary finite interval, in terms of a
 Fredholm determinant of the solution operator of a certain boundary value problem. The
 formula is obtained as a fine mesh limit of an older formula of \citet{prahoferSpohn} for
 the $n$-dimensional distributions (see \eqref{eq:airy2fd} below). The advantage of the
 alternative formula for variational analysis is that its complexity is no longer
 diverging with the number of spatial points.  Using this formula, we were able to give a
 direct proof of Johansson's result \cite{Johansson}, study the half line version
 \cite{qr-airy1to2}, and derive an exact formula for the probability density of the argmax
 of the Airy$_2$ process minus a parabola, the {\it polymer endpoint distribution}
 \cite{mqr}.
 
%  It was quickly pointed out that the continuum formula for the Airy$_2$ process must be
%  some sort of pathwise version of the Karlin-McGregor formula, and this has now been made
%  precise \cite{bcr}. On the other hand, if one understands this as a reason for the
%  existence of such a formula, it suggests that there should {\it not} be an analogous
%  formula for the Airy$_1$ process. \note{I would rewrite this...}

% \note{The continuum and discrete formulas are being mixed a bit in the discussion}
 
% It remains to understand the
 % origin of such formulas, or if the Airy$_1$ formula may arise from some sort of {\it
 %   generalized} line ensemble \note{...}.  Since we have an earlier continuum formula for
 % the Airy$_2$ process, we describe how the proof can be modified to obtain a direct proof
 % of the regularity there, using the Kolmogorov condition for the two dimensional
 % distributions.  Note that for the Airy$_2$ process, there are recent alternative
 % approaches using line ensembles \cite{corwinHammond}.

 In this article we will obtain analogous discrete and continuum formulas for the Airy$_1$
 process, and use them to prove directly that it is H\"older $\frac12-\delta$ continuous
 for any $\delta>0$. This regularity of Airy$_1$ is expected from the fact that the
 process is believed to look locally like a Brownian motion. In fact, we will show in this
 direction, using the alternative determinantal formula, that the finite dimensional distributions
 of the Airy$_1$ process converge under diffusive scaling to those of a Brownian motion.
 
 Note that the existence of formulas for the Airy$_1$ process involving boundary value
 operators is to some extent surprising. In the case of the Airy$_2$ process, which is the
 limit of the rescaled top line in a system of non-intersecting Brownian motions (Dyson's
 Brownian motion for the Gaussian Unitary Ensemble), the formula can be seen as a certain
 extension of the Karlin-McGregor formula (see \cite{bcr}). On the other hand, there is no
 known analogous construction of the Airy$_1$ process (see in particular
 \cite{borFerrPrah}), for which the associated determinantal process is signed (see
 \cite{bfp}), and thus it is not at all apparent where formulas like \eqref{eq:airy1fd} or
 \eqref{eq:basic1} below are coming from.

 % they have a direct interpretation in terms of the top line of non-intersecting paths and the
 % Karlin-McGregor formula. In the Airy$_1$ case, such an interpretation, if it exists,
 % would have to be more indirect. \note{Mention the signed determinantal process stuff...}

\subsection{Statement of the results}

Now we turn to a precise description of the Airy$_1$ process, which will be our main object
of study. It was first derived by \citet{sasamoto} (see also \cite{borFerPrahSasam,bfp})
by asymptotic analysis of exact formulas for TASEP with periodic initial data. It is a
stationary process defined through its finite-dimensional distributions, given by a
determinantal formula: for $x_1,\dots,x_n\in\mathbb{R}$ and $t_1<\dots<t_n$ in
$\mathbb{R}$,
\begin{equation}\label{eq:detAiry1}
\pp\!\left(\aipo(t_1)\le x_1,\dots,\aipo(t_n)\le x_n\right) = 
\det(I-\mathrm{f}^{1/2}K^{\mathrm{ext}}_1\mathrm{f}^{1/2})_{L^2(\{t_1,\dots,t_n\}\times\mathbb{R})},
\end{equation}
where we have counting measure on $\{t_1,\dots,t_n\}$ and
Lebesgue measure on $\mathbb{R}$,  $\mathrm f$ is defined on
$\{t_1,\dots,t_n\}\times\mathbb{R}$ by $\mathrm{f}(t_j,x)=\uno{x\in(x_j,\infty)}$ and
\begin{multline}\label{eq:fExtAiry1}
K^{\rm ext}_1(t,x;t',x')=-\frac{1}{\sqrt{4\pi
    (t'-t)}}\exp\!\left(-\frac{(x'-x)^2}{4 (t'-t)}\right)\uno{t'>t}\\
+\Ai(x+x'+(t'-t)^2) \exp\!\left((t'-t)(x+x')+\frac23(t'-t)^3\right).
\end{multline}
Here, and in everything that follows, the determinant means the Fredholm
determinant in the Hilbert space indicated in the subscript. In particular from
\eqref{eq:fExtAiry1} and \cite{ferrariSpohn} one obtains that the one-point distribution
of the Airy$_1$ process is given in terms of the Tracy-Widom largest eigenvalue distribution for the Gaussian
Orthogonal Ensemble (GOE) \cite{tracyWidom2}:
\[\pp(\aipo(0)\leq m)=F_{\rm GOE}(2m).\]
Note that it follows from \eqref{eq:detAiry1} that $\aipo(t)$ has the same distribution as $\aipo(-t)$.

The definition of the Airy$_1$ process is analogous to that of the Airy$_2$ process,
introduced by \citet{prahoferSpohn}, whose $n$ dimensional distributions are given by
\begin{equation}\label{eq:fExtAiry2}
  \mathbb{P}\!\left(\aip(t_1)\le x_1,\dots,\aip(t_n)\le x_n\right) = 
  \det(I-\mathrm{f}^{1/2}K_2^{\mathrm{ext}}\mathrm{f}^{1/2})_{L^2(\{t_1,\dots,t_n\}\times\mathbb{R})},
\end{equation}
where the {\it extended Airy kernel} \cite{prahoferSpohn,FNH,macedo} $K_2^{\mathrm{ext}}$ is
defined by
\[K_2^\mathrm{ext}(t,x;t',x')=
\begin{cases}
\int_0^\infty d\lambda\,e^{-\lambda(t-t')}\Ai(x+\lambda)\Ai(x'+\lambda), &\text{if $t\ge t'$}\\
-\int_{-\infty}^0 d\lambda\,e^{-\lambda(t-t')}\Ai(x+\lambda)\Ai(x'+\lambda),  &\text{if $t<t'$},
\end{cases}.\]

The analogy between the definitions becomes clearer in light of the following
observations. Letting $K_{\Ai}$ denote the \emph{Airy kernel}
\[K_{\Ai}(x,y)=\int_{-\infty}^0 d\lambda\Ai(x-\lambda)\Ai(y-\lambda)\] and $H$ denote the
\emph{Airy Hamiltonian}
\[H=-\Delta+x,\] where $\Delta=\p_x^2$ denotes the one-dimensional Laplacian, one
can show (formally) that the extended Airy kernel can be rewritten as
\begin{equation}
K^{\rm ext}_2(t,x;t',x')
=-e^{-(t'-t)H}(x,x')\uno{t'>t}+e^{tH}\K e^{-t'H}(x,x').\label{eq:extAiry2}
\end{equation}
On the other hand, as shown in Appendix A of \cite{borFerPrahSasam}, $K^{\rm ext}_1$ can be expressed
(formally) in the following alternative way:
\begin{equation}\label{eq:extAiry1}
K^{\rm ext}_1(t,x;t',x')=-e^{(t'-t)\Delta}(x,x')\uno{t'>t}+e^{-t\Delta}B_0e^{t'\Delta}(x,x'),
\end{equation}
where
\[B_0(x,y)=\Ai(x+y).\] Note that \eqref{eq:extAiry1} corresponds exactly to
\eqref{eq:extAiry2} after replacing $H$ by $-\Delta$ and $\K$ by $B_0$. This particular
replacement was emphasized in \cite{ferrariReview}; more generally, all the extended
kernels arising in this and related areas have an analogous structure. We stress that
both \eqref{eq:extAiry2} and \eqref{eq:extAiry1} should be regarded at this point as
formal identities, as it is not clear how to make sense of $e^{-tH}$ and $e^{t\Delta}$ for
$t<0$.

Our first result provides a new determinantal formula for the finite-dimensional
distributions of the Airy$_1$ process without using extended kernels or, in other words,
involving the Fredholm determinant of an operator acting on $L^2(\rr)$ instead of
$L^2(\{t_1,\dotsc,t_n\}\times\rr)$. For the Airy$_2$ process such a formula was introduced by
\citet{prahoferSpohn} as its original definition:
\begin{multline}
  \label{eq:airy2fd}
    \pp\!\left(\aip(t_1)\leq x_1,\dotsc,\aip(t_n)\leq x_n\right)\\
  =\det\!\left(I-K_{\Ai}+\bar P_{x_1}e^{(t_1-t_2)H}\bar P_{x_2}e^{(t_2-t_3)H}\dotsm
   \bar P_{x_n}e^{(t_n-t_1)H}K_{\Ai}\right)_{L^2(\rr)},
\end{multline}
where $\bar P_a$ denotes projection onto the interval $(-\infty,a]$. The equivalence of
\eqref{eq:fExtAiry2} and \eqref{eq:airy2fd} was derived in
\cite{prahoferSpohn,prolhacSpohn}, see Remarks \ref{rem:airy2case1} and
\ref{rem:airy2case2} below for a discussion about some technical details. Our result
states that the finite-dimensional distributions of the Airy$_1$ process admit the same
representation after replacing $H$ by $-\Delta$ and $\K$ by $B_0$.

\begin{thm}\label{thm:airy1}
  The finite-dimensional distributions of the Airy$\hspace{0.1em}_1$ process are given by the following formula: for
  $x_1,\dots,x_n\in\mathbb{R}$ and $t_1<\dots<t_n$ in $\mathbb{R}$,
  \begin{multline}
    \label{eq:airy1fd}
    \pp\!\left(\aipo(t_1)\leq x_1,\dotsc,\aipo(t_n)\leq x_n\right)\\
    =\det\!\left(I-B_0+\bar P_{x_1}e^{-(t_1-t_2)\Delta}\bar P_{x_2}e^{-(t_2-t_1)\Delta}\dotsm
      \bar P_{x_n}e^{-(t_n-t_1)\Delta}B_0\right)_{L^2(\rr)}.
  \end{multline}
\end{thm}

\begin{rem}\label{rem:t}\mbox{}
  \begin{enumerate}[label=\arabic*.]
  \item Note that, since $t_1<\dots<t_n$, all the heat kernels in \eqref{eq:airy1fd} are
    well defined except for the first one. The same situation is present in the formula for
    the Airy$_2$ process, as the factor $e^{(t_n-t_1)H}$ in \eqref{eq:airy2fd} is in
    principle ill-defined. The situation is resolved in that case by observing that
    $e^{(t_n-t_1)H}$ is applied after $\K$ in \eqref{eq:airy2fd}, and $\K$ is a projection
    operator onto the negative eigenspace of $H$. In our case the situation is resolved by
    Proposition \ref{prop:mLaplacian} below.
 \item The operator
    \[J:=-B_0+\bar P_{x_1}e^{-(t_1-t_2)\Delta}\bar P_{x_2}e^{-(t_2-t_3)\Delta}\dotsm \bar
    P_{x_n}e^{-(t_n-t_1)\Delta}B_0\] appearing inside the determinant in
    \eqref{eq:airy1fd} is not trace class, basically because the heat kernel is not even
    Hilbert-Schmidt. However, we will show in Proposition \ref{prop:traceclass} that there
    is a conjugate operator $\wt J=U^{-1} JU$ which is trace class in $L^2(\rr)$, so the
    formula \eqref{eq:airy1fd} should be computed as $\det(I-\wt
    J)_{L^2(\rr)}$. Alternatively, this implies that the Fredholm determinant in
    \eqref{eq:airy1fd} regarded as its Fredholm expansion series is well defined. (The same
    issue arises in \eqref{eq:detAiry1}, as $K^{\rm ext}_1$ is not trace class on
    $L^2(\{t_1,\dots,t_n\}\times\rr)$; this is resolved in Appendix A of \cite{bfp}).
  \item Note that the issue discussed in the last point does not arise in the formula
    \eqref{eq:airy2fd} for the Airy$_2$ process. The fact that the operator appearing in
    that formula is trace class is proved in Proposition \gref{prop:theta}.
  \end{enumerate}
\end{rem}

The following result shows that we are allowed to consider the operator $e^{-t\Delta}$ for
$t>0$ as long as it is applied after $B_0$.

\begin{prop}\label{prop:mLaplacian}
  For fixed $t,y\in\rr$ let
  $\varphi_{t,y}(x)=e^{-2t^3/3-(x+y)t}\Ai(x+y+t^2)$.
  Then for all $s,t>0$ we have
  \begin{equation}
    e^{s\Delta}\varphi_{t,y}(x)=\varphi_{t-s,y}(x).\label{eq:varphisg}
  \end{equation}
  In particular,
  $e^{t\Delta}\varphi_{t,y}=\Ai(x+y)$, and as a consequence the kernel $e^{-t\Delta}B_0$ is
  well defined for every $t>0$ via the formula
  \begin{equation}
    e^{-t\Delta}B_0=e^{-2t^3/3-(x+y)t}\Ai(x+y+t^2)\label{eq:e-tDB0}
  \end{equation}
  and it satisfies the group property in the sense that
  $e^{(s+t)\Delta}B_0=e^{s\Delta}e^{t\Delta}B_0$ for all $s,t\in\rr$.
\end{prop}

We remark that versions of the above identities appear in earlier works on the Airy$_1$
process, and in particular in \cite{sasamoto,bfp,borFerPrahSasam}. Proposition
\ref{prop:mLaplacian} allows us to make sense of \eqref{eq:extAiry1}: since the Airy$_1$
process is stationary, by shifting $t_1,\dotsc,t_n$ we may assume that $0<t_1<\dots<t_n$,
and then all the heat kernels with a negative parameter in \eqref{eq:extAiry1} appear
applied after $B_0$. The same type of argument allows to make sense of
\eqref{eq:extAiry2}, \eqref{eq:airy2fd} and \eqref{eq:airy1fd} (though see also the last
paragraph of Remark \ref{rem:airy2case2}).

As we mentioned, formulas \eqref{eq:airy2fd} and \eqref{eq:airy1fd} are better adapted
than the standard extended kernel formulas to short range properties of the process. As a first
application we will prove

\begin{thm}\label{thm:regularity}
  The Airy$\hspace{0.05em}_1$ process $\aipo$ and the Airy$\hspace{0.05em}_2$ process
  $\aip$ have versions with H\"older continuous paths with exponent $\tfrac12-\delta$ for
  any $\delta>0$.
\end{thm}

Recall that continuity was known for $\aip$ but not for $\aipo$. The H\"older $\frac12-$
continuity for $\aip$ also follows from recent work of \citet{corwinHammond}. They study
the Airy line ensemble directly, obtaining the continuity (and H\"older $\frac12-$
continuity) directly from a certain Brownian Gibbs property. In general, all the Airy
processes are supposed to be locally Brownian. Note that the definition of
locally Brownian is not unique. For $\aip$ it follows from \cite{corwinHammond} that it is
locally absolutely continuous with respect to Brownian motion. Analogous results have
recently become available for the solutions of the KPZ equation at finite times
\cite{Hairer,qrLocalBrownianKPZ,corwinHammondKPZ}. For $\aipo$ the line ensemble picture
is missing at the present time, so a proof was lacking.  As another application of the
formulas, we prove that the Airy$_1$ process is locally Brownian in the sense that under
local Brownian scaling, the incremental process converges to that of Brownian motion.

\begin{thm}\label{thm:localBM}  For any fixed $s\in\mathbb{R}$, 
  let $B_\ep (\cdot)$ be defined by $B_\ep(t)= \ep^{-1/2}(\aipo(s+ \ep t)-\aipo(s) )$,
  $t>0$.  Then $B_\ep (\cdot)$ converges to Brownian motion in the sense of convergence of
  finite dimensional distributions.  The same holds for $\tilde{B}_\ep (\cdot)$ defined by
  $\tilde{B}_\ep(t)= B_\ep(-t)$, $t>0$.
\end{thm}

Note that by stationarity there is no loss of generality in taking $s=0$ in the theorem,
while the statement about $\tilde{B}_\ep (\cdot)$ follows from the statement about $B_\ep
(\cdot)$ by time reversal invariance of Airy$_1$. The analogue of Theorem
\ref{thm:localBM} for Airy$_2$, which follows from its local absolute continuity with
respect to Brownian motion, was proved earlier by \citet{hagg}, and can also be obtained
directly by our method. We remark also that, using an analogue of \eqref{eq:airy1fd} for
the Airy$_{2\to1}$ process, which will appear in upcoming work \cite{bcr}, it should not
be hard to adapt our proofs to show that $\mathcal{A}_{2\to1}$ is H\"older $\frac12-$
continuous and is locally Brownian in the sense of the last result (in fact, the result
of \cite{bcr} is more general and should allow one to extend our proofs to other processes).

Going back to $\aipo$, one can be quite precise in terms of finite dimensional
distributions. Letting $0<t_1<\cdots < t_n$,  we will prove that
% \begin{align}\label{i}
%  &\pp\!\left(\aipo(\ep t_1)\leq x+\sqrt{\ep}y_1 , \ldots, \aipo(\ep t_n)\leq
%    x+\sqrt{\ep}y_n\,\middle|\,\aipo(0)=x \right)\\
%  & =\ee[ \uno{B(t_i)\le y_i, i=1,\ldots,n}  ~~g^\ep_{t_1,y_1,\ldots,t_n,y_n} (x, B(t_n))] h^\ep_{t_1,y_1,\ldots,t_n,y_n} (x)
%, 
%  \end{align}
 \begin{multline}\label{i}
   \pp\!\left(\aipo(\ep t_1)\leq x+\sqrt{\ep}y_1 , \ldots, \aipo(\ep t_n)\leq
    x+\sqrt{\ep}y_n\,\middle|\,\aipo(0)=x \right)\\
   =\ee\!\left( \uno{B(t_i)\le y_i, i=1,\ldots,n}\,g^\ep_{\mathbf{t},\mathbf{y}} (x, B(t_n))\right) h^\ep_{\mathbf{t},\mathbf{y}} (x), 
  \end{multline}
where $B(t)$ is a standard Brownian motion with $B(0)=0$ and
\begin{equation}\label{II}
  g^\ep_{\mathbf{t},\mathbf{y}} (x, z)= 
  \frac{ \int_{-\infty}^\infty du\, e^{-\ep t_n\Delta }
    B_0(\sqrt{\ep}z+x,u) \left(I-B_0 + \Lambda_{(0,\ep\mathbf{t})
      }^{(x,\sqrt{\ep}\mathbf{y}+x)}
      e^{- \ep t_n\Delta} B_0\right)^{-1}\!(u,x)}{
    \int_{-\infty}^\infty du\, B_0(x,u)\left(I-B_0+\bar{P}_x B_0\right)^{-1}\!(u,x)},
\end{equation}
where
$
\Lambda_{ (0,\ep\mathbf{t})
      }^{(x,\sqrt{\ep}\mathbf{y}+x)}= \bar P_{x}e^{t_1\Delta}\bar P_{y_1+x}e^{(t_2-t_1)\Delta}\dotsm
      e^{(t_n-t_{n-1})\Delta}\bar P_{y_n+x}
$
and
\begin{equation}\label{I}
  h^\ep_{\mathbf{t},\mathbf{y}} (x, z)=   \frac{\pp\!\left(\aipo(0)\le x , \aipo( \ep t_1)\leq x+\sqrt{\ep}y_1
      , \ldots, \aipo( \ep t_n)\leq x+\sqrt{\ep}y_n\right) }{F_{\rm GOE}(2x)}.
\end{equation}
One has
\begin{equation}
  \lim_{\ep\to0} g^\ep_{\mathbf{t},\mathbf{y}} (x, z) =  \lim_{\ep\to0} h^\ep_{\mathbf{t},\mathbf{y}} (x) = 1,
\end{equation}
from which it follows from \eqref{i} that the finite dimensional distributions converge to those of Brownian 
motion.  
%The H\"older continuity then implies the tightness in $C[0,T]$ to obtain the stronger convergence in the theorem.
It would be interesting to understand the role of $g^\ep_{\mathbf{t},\mathbf{y}}(x, z)$.
Expansions $g^\ep_{\mathbf{t},\mathbf{y}} (x, z)= 1 + \ep^{1/2}
g^{(1)}_{\mathbf{t},\mathbf{y}}(x, z)+ \mathcal{O}(\ep)$ and
$h^\ep_{\mathbf{t},\mathbf{y}} (x)= 1 + \ep^{1/2} h^{(1)}_{\mathbf{t},\mathbf{y}} (x)+
\mathcal{O}(\ep)$ may identify the infinitesimal increments of $\aipo$ in order to develop
a stochastic calculus.

One of course has formulas analogous to \eqref{i} for the Airy$_2$ process (and, in view
of \cite{bcr}, other processes such as Airy$_{2\to1}$), but we do not include them here.

%
%\begin{multline}
%  \lim_{\ep\to0}  \pp\!\left(\aipo(s+ \ep t_1)-\aipo(s)\leq \sqrt{\ep} y_1 , \ldots,
%    \aipo(s+ \ep t_n)-\aipo(s)\leq \sqrt{\ep} y_n~\mid~\aipo(s) \right) \\
%  =  \pp\!\left( B(t_1) \le y_1,\ldots, B(t_n) \le y_n\right),
%\end{multline}
%where $B(t)$ is a standard Brownian motion with $B(0)=0$.  Furthermore, $\tilde{\aipo}^\ep(t)=  \sqrt{\ep}(\aipo(s\pm \ep t)-\aipo(s) )$ converges to Brownian motion (in the sense of weak convergence of probability measures on the space of continuous functions equipped with the topology of uniform convergence on compact sets).
%\end{thm}
%
%Actually, one even gets a formal expression for the next order correction, which is of size $\sqrt{\ep}$.   From (\ref{II}) we have
%\begin{multline}
% \ep^{-1/2}[ \pp\!\left(\aipo(s+ \ep t_i)-\aipo(s)\leq \sqrt{\ep} y_i , i=1,\ldots,n~\mid~\aipo(s)=x \right) - \pp\!\left( B(t_i) \le y_i, i=1,\ldots,n\right)] \\= \frac{F'_{\rm GOE}(x)}{F_{\rm GOE}(x)}
%  g_{t_1,\ldots,t_n}(y_1,\ldots,y_n)   
%  + h(x) \int zdz \pp\!\left( B(t_1) \le y_1+ \tfrac{t_1}{t_n} z,\ldots, B(t_n) \le y_n+z\right) + \mathcal{O}(\ep^{1/2}) .\end{multline}
%where
%\begin{equation}
% g_{t_1,\ldots,t_n}(y_1,\ldots,y_n) :=  \int_0^\infty (1- P_0( B(t_i) \le y_i +z , i=1,\ldots,n) ) dz\pp\!\left( B(t_i) \le y_i, i=1,\ldots,n\right)
% \end{equation}
% and 
%\begin{equation}
%h(x) := 
%\frac{ \int_{-\infty}^\infty du\,
%    B'_0(x,u) \left(I-B_0+\bar{P}_x B_0\right)^{-1}(u,x)}{
%    \int_{-\infty}^\infty du\, B_0(x,u)\left(I-B_0+\bar{P}_x B_0\right)^{-1}(u,x)}.
%    \end{equation}

Our last result, which is an application of Theorems \ref{thm:airy1} and \ref{thm:regularity},
gives a determinantal formula for the continuum statistics of the Airy$_1$ process on a
finite interval. This was done for the Airy$_2$ process in \cite{cqr}, and the same
argument will allow us to take a limit of the formula in Theorem \ref{thm:airy1} as the
size of the mesh in $t$ goes to 0.

Fix $\ell<r$. Given $g\in H^1([\ell,r])$ (i.e. both $g$ and
its derivative are in $L^2([\ell,r])$), define an operator $\Lambda^g_{[\ell,r]}$ acting
on $L^2(\rr)$ as follows: $\Lambda^g_{[\ell,r]}f(\cdot)=u(r,\cdot)$, where $u(r,\cdot)$ is
the solution at time $r$ of the boundary value problem
 \begin{equation}
\begin{aligned}
  \p_tu-\Delta u&=0\quad\text{for }x<g(t), \,\,t\in (\ell,r)\\
  u(\ell,x)&=f(x)\uno{x<g(\ell)}\\
  u(t,x)&=0\quad\text{for }x\ge g(t).
\end{aligned}\label{eq:bdval}
\end{equation}
The fact that this problem makes sense for $g\in H^1([\ell,r])$ is not hard and can be
seen from the proof of Proposition \ref{prop:traceclass} below (see also Proposition \gref{prop:theta}).

\begin{thm}\label{thm:aiLo}
  \begin{equation}\label{eq:basic1}
    \pp\!\left(\aipo(t)\leq g(t)\text{ for
      }t\in[\ell,r]\right)=\det\!\left(I-B_0+\Lambda^g_{[\ell,r]}e^{-(r-\ell)\Delta}B_0\right)_{L^2(\rr)}.
  \end{equation}
\end{thm}

In other words, {\it hitting probabilities of curves by $\aipo$ can be expressed in terms of Fredholm determinants
of the analogous hitting probabilities for Brownian motion.}  

One can check easily using the Feynman-Kac formula that the kernel of
$\Lambda^g_{[\ell,r]}$ has the following form:
\begin{equation}
\Lambda^g_{[\ell,r]}(x,y)=\frac{e^{-(x-y)^2/4(r-\ell)}}{\sqrt{4\pi(r-\ell)}}
  \pp_{\hat b(\ell)=x,\hat b(r)=y}\!\left(\hat b(s)\leq g(s)\text{ on }[\ell,r]\right),\label{eq:Lambda}
\end{equation}
where the probability is computed with respect to a Brownian bridge $\hat b(s)$ from $x$
at time $\ell$ to $y$ at time $r$ and with diffusion coefficient $2$. We remark that the
kernel $-B_0+\Lambda^g_{[\ell,r]}e^{-(r-\ell)\Delta}B_0$ is not trace class, but as in the
discrete case (see Remark \ref{rem:t}) we will show that there is conjugate operator which
is, see Proposition \ref{prop:traceclass}.

The corresponding formula for the Airy$_2$ process, provided in Theorem \gref{thm:aiL}, is
the same as \eqref{eq:basic1} after replacing $-\Delta$ by $H$ and $B_0$ by $\K$. The
corresponding boundary value operator $\Theta^g_{[\ell,r]}$ in that case is actually more
complicated than $\Lambda^g_{[\ell,r]}$, as in our case there is no potential term in the
partial differential equation in \eqref{eq:bdval}. 
%\note{I think we can erase the rest
%  now} As we mentioned earlier, a formula for the Airy$_2$ process given in terms of the
%determinant of an operator involving a boundary value problem for Brownian motion is not
%so unexpected, as $\aip$ is the limit of the rescaled top line in a system of
%non-intersecting Brownian motions (Dyson's Brownian motion for the Gaussian Unitary
%Ensemble), so the Karlin-McGregor formula suggests that such a formula could hold (see
%\cite{bcr}). On the other hand, there is no known analogous construction of the Airy$_1$
%process (in particular, see \cite{borFerrPrah}), for which the associated determinantal
%process is signed, see \cite{bfp}, and thus it is not at all apparent where a formula like
%\eqref{eq:basic1} (or even \eqref{eq:airy1fd}) is coming from.

\vs

\paragraph{\bf Acknowledgements}
Both authors were supported by the Natural Science and Engineering Research Council of
Canada, and DR was supported by a Fields-Ontario Postdoctoral Fellowship and by Fondecyt
Grant 1120309. The authors thank Ivan Corwin and Alexei Borodin for interesting and useful
conversations.

\section{Proof of the determinantal formula}
\label{sec:proofs}

Throughout this section and the next we will denote by $\|\cdot\|_1$ and $\|\cdot\|_2$
respectively the trace class and Hilbert-Schmidt norms of operators on $L^2(\rr)$ (see
Section \gref{sec:aiL} for the definitions or \cite{simon} for a complete treatment).

\begin{proof}[Proof of Proposition \ref{prop:mLaplacian}]
  Recall that $\Ai(z)=(2\pi\I)^{-1}\int_{\Gamma_c} du\,e^{u^3/3-uz}$, where
  $\Gamma_c=\{c+\I y,\,y\in\rr\}$ for
  any fixed $c>0$. Then
  \[e^{s\Delta}\varphi_{t,y}(x)=\frac{1}{2\pi\I}\int_{-\infty}^\infty dz\int_{\Gamma_{c}} du
  \frac{1}{\sqrt{4\pi s}}e^{-(x-z)^2/4s-2t^3/3-(z+y)t+u^3/3-u(z+y+t^2)}.\]
  We can compute the $z$ integral first, which is just a Gaussian integral, to obtain
   \[e^{s\Delta}\varphi_{t,y}(x)=\frac{1}{2\pi\I}\int_{\Gamma_{c}} du\,e^{\frac{1}{3}
       (t+u) ((3 s-2 t+u) (t+u)-3 (x+y))}.\]
  Shifting $u$ to $u-s$ we get
  \[e^{s\Delta}\varphi_{t,y}(x)=\frac{1}{2\pi\I}\int_{\Gamma_{c+s}}
  e^{u^3/3-u(x+y+(t-s)^2)-(x+y)(t-s)-2(t-s)^3/3}=\varphi_{t-s,y}(x),\]
  which proves \eqref{eq:varphisg}. The remaining statements in the proposition follow
  directly from this identity.
\end{proof}

We turn now to the proof of Theorem \ref{thm:airy1}. The argument is based on the
derivation of the equivalence of \eqref{eq:fExtAiry2} and \eqref{eq:airy2fd} for the
Airy$_2$ case given by \citet{prolhacSpohn}, and in fact the algebraic procedure we will
use is basically equivalent to theirs. In the case of the Airy$_1$ process one has to make sure
throughout the proof that the algebraic manipulations are being done on operators which
are trace class, so that the Fredholm determinants considered are well defined. This is
done by rewriting the algebraic procedure of \cite{prolhacSpohn} so that in each step one
can conjugate by the correct operators and check that the resulting conjugated
operators are trace class as needed.

\begin{rem}\label{rem:airy2case1}
  Our proof of Theorem \ref{thm:airy1} can be used to complete the details and
  provide all the necessary justifications in the proof given in \cite{prolhacSpohn} for
  the Airy$_2$ case. In one sense the argument in that case is simpler, because the
  kernels in \eqref{eq:fExtAiry2} and \eqref{eq:airy2fd} are already trace
  class. Nevertheless the Airy$_2$ case presents an additional difficulty, namely that
  even for $t>0$ the operator $e^{-tH}$ does not map $L^2(\rr)$ into itself (note that
  this issue does not arise in the Airy$_1$ case, as $e^{t\Delta}$ is clearly a bounded operator
  acting on $L^2(\rr)$ for $t>0$). We will explain in Remark \ref{rem:airy2case2} how this
  can be addressed, and in particular how the proof below has to be changed to provide a
  rigorous proof for the Airy$_2$ case.
\end{rem}

\begin{proof}[Proof of Theorem \ref{thm:airy1}]
  We will retain most of the notation of \cite{prolhacSpohn}, and as in that paper we use
  sans-serif fonts (e.g. $\sT$) for operators on
  $L^2(\{t_1,\dots,t_n\}\times\mathbb{R})$. Such an operator can be regarded as an
  operator-valued matrix $\big(\sT_{i,j}\big)_{i,j=1,\dots,n}$ with entries $\sT_{i,j}\in
  L^2(\rr)$ acting on $f\in L^2(\rr)^{n}$ as $(\sT f)_i=\sum_{j=1}^n\sT_{i,j}f_j$ (or,
  more precisely, as an operator acting on $\rr^{n}\otimes L^2(\rr)$). We will use serif
  fonts for the matrix entries (e.g. $\sT_{i,j}=T$ for some $T\in L^2(\rr)$). All
  determinants throughout this proof are computed on $L^2(\{t_1,\dots,t_n\}\times\rr)$
  unless otherwise indicated.

  Recall from Proposition \ref{prop:mLaplacian} that $e^{t\Delta}B_0$ satisfies the
  semigroup property $e^{s\Delta}e^{t\Delta}B_0=e^{(s+t)\Delta}B_0$ for all
  $s,t\in\rr$. We will use this fact several times below. We will also use the fact that,
  since $B_0(x,y)$ depends only on $x+y$, $e^{t\Delta}$ and $B_0$ commute for
  $t>0$. Finally, as explained after the proof of Proposition \ref{prop:mLaplacian}, we
  may (and will) assume that $t_i>0$ for $i=1,\dots,n$.

  Let $\sK={\rm f}^{1/2}K^1_{\rm ext}{\rm f}^{1/2}$, with $K^1_{\rm ext}$ defined through
  \eqref{eq:extAiry1} and f as in \eqref{eq:detAiry1}. Using the above interpretation
  $\sK$ can be written as
  \begin{equation}\label{eq:sK}
    \sK=\sP(\sT^{-}\sK^0+\sT^{+}(\sK^0-\sI))\sP,
  \end{equation}
  where
  \begin{equation}
    \sK^{0}_{ij}=B_0\uno{i=j},\quad \sP_{i,j}=P_{x_{j}}\uno{i=j},
  \end{equation}
  with $P_a=I-\bar P_a$ denoting projection onto the interval $[a,\infty)$, and $\sT^{-}$,
  $\sT^{+}$ are lower triangular, respectively strictly upper triangular, and defined by
  \begin{equation}
    \sT^{-}_{ij} = e^{-(t_{i}-t_{j})\Delta}\uno{i\geq j},\quad
    \sT^{+}_{ij}=e^{-(t_{i}-t_{j})\Delta}\uno{i < j}.
  \end{equation}
  Observe that the all heat kernels in $\sT^+$ have positive parameters, while those in
  $\sT^-$ have negative parameters but appear applied after $B_0$ in the expression for
  $\sK$ in \eqref{eq:sK}, so Proposition \ref{prop:mLaplacian} ensures that \eqref{eq:sK}
  makes sense.

  As we mentioned in Remark \ref{rem:t}, it is proved in \cite{bfp} that there is an
  invertible operator $\sV$ such that $\sV\sK\sV^{-1}$ is trace class. Explicitly, $\sV$
  is a (diagonal) multiplication operator given by
  \[\sV_{i,j}=V_i\uno{i=j}\qquad\text{with}\quad V_if(x)=(1+x^2)^{-2i}f(x).\]
  Since $\sV\sP\sT^+\sP\sV^{-1}$ is strictly upper triangular, $\sI+\sV\sP\sT^+\sP\sV^{-1}$ is
  invertible, and then we can write
  \begin{equation}
    \label{eq:detDec}
    \det\!\big(\sI-\sV\sK\sV^{-1}\big)
    =\det\!\big((\sI+\sWo)(\sI-(\sI+\sWo)^{-1}\sWt)\big)
  \end{equation}
  with
  \begin{equation}
    \sWo=\sV\sP\sT^+\sP\sV^{-1},\qquad\sWt=\sV\sP(\sT^-+\sT^+)\sK^0\sP\sV^{-1}.\label{eq:sWs}
  \end{equation}
  We remark that $\sWo$ is trace class by Lemma A.2 in \cite{bfp}.

  Next we want to obtain an explicit expression for $(\sI+\sWo)^{-1}\sWt$. Observe that
  \begin{equation}\label{eq:IT+}
    \big[(\sI+\sT^{+})^{-1})\big]_{i,j}=\uno{i=j}-e^{-(t_{i}-t_{i+1})\Delta}\uno{i=j-1},
  \end{equation}
  which can be checked directly using the semigroup property of the heat kernel. In
  particular $\sI+\sT^+$ is invertible, so we can write
  \begin{equation}
    (\sI+\sWo)^{-1}\sWt=(\sI+\sWo)^{-1}\sV\sP(\sT^-+\sT^+)(\sI+\sT^+)^{-1}\sK^0(\sI+\sT^+)\sP\sV^{-1},\label{eq:decsW}
  \end{equation}
  where we have used the fact that $e^{t\Delta}$ and $B_0$ commute for $t>0$, and hence so
  do $\sT^+$ and $\sK^0$. Using \eqref{eq:IT+} we deduce that 
  \begin{equation}
    \begin{aligned}
      \big[(\sT^{-}+\sT^{+})(\sI+\sT^{+})^{-1}\sK^0\big]_{i,j}
      &=e^{-(t_i-t_j)\Delta}B_0-e^{-(t_i-t_{j-1})\Delta}e^{-(t_{j-1}-t_j)\Delta}B_0\uno{j>1}\\
      &=e^{-(t_{i}-t_{1})\Delta}B_0\uno{j=1}.
    \end{aligned}\label{eq:T-T+}
  \end{equation}
  Note that only the first column of this matrix has non-zero entries.

  Observe now that, since $\sV\sP\sT^+\sP\sV^{-1}$ is strictly upper triangular, we have
  $(\sV\sP\sT^+\sP\sV^{-1})^{n+1}=0$, which implies that
  \begin{equation}
    (\sI+\sWo)^{-1}=\sum_{k=0}^n(-1)^k(\sV\sP\sT^{+}\sP\sV^{-1})^k.\label{eq:inv}
  \end{equation}
  On the other hand by \eqref{eq:T-T+} we have for $0\leq k\leq n-i$
  \begin{multline}\label{eq:termInv}
    \left[(\sV\sP\sT^+\sP\sV^{-1})^k\sV\sP(\sT^-+\sT^+)(\sI+\sT^+)^{-1}\sK^0\right]_{i,1}\\
    =\quad\smashoperator{\sum_{i<a_1<\dots<a_k\leq n}}\quad
    V_iP_{x_{i}}e^{-(t_i-t_{a_1})\Delta}P_{x_{a_1}}e^{-(t_{a_{1}}-t_{a_2})\Delta} \dotsm
    P_{x_{a_{k-1}}}e^{-(t_{a_{k-1}}-t_{a_k})\Delta}P_{x_{a_k}}e^{-(t_{a_k}-t_{1})\Delta}B_0,
  \end{multline}
  which follows from \eqref{eq:T-T+} and the definition of $\sP\sT^+\sP$, while for
  $k>n-i$ the left side above equals 0 (and the case $k=0$ is interpreted as
  $V_iP_{x_i}e^{-(t_i-t_1)\Delta}B_0$). Replacing each factor $P_{x}$ except the first one by
  $I-\bar P_{x}$ and using the semigroup property for the heat kernel we deduce that the
  last expression equals
  \begin{multline}
    \sum_{m=0}^{k}\sum_{i=b_0<b_1<\dots<b_m\leq n}\binom{n-i-m}{k-m}(-1)^{m}V_{b_0}
    P_{x_{b_0}}e^{-(t_{b_0}-t_{{b_1}})\Delta}\bar
    P_{x_{b_1}}e^{-(t_{b_1}-t_{b_2})\Delta}\\
    \hspace{2.5in}\dotsm\bar P_{x_{{b_{m-1}}}}e^{-(t_{b_{m-1}}-t_{b_m})\Delta} \bar
    P_{x_{b_m}}e^{-(t_{b_m}-t_{1})\Delta}B_0.
  \end{multline}
  Summing the above expression times $(-1)^k$ from $k=0$ to $k=n-i$ and interchanging the
  order of summation leads to
  \begin{multline}
    \sum_{m=0}^{n-i}\sum_{k=m}^{n-i}\sum_{i=b_0<b_1<\dots<b_m\leq
      n}\binom{n-i-m}{k-m}(-1)^{k+m}V_{b_0}P_{x_{b_0}}e^{-(t_{b_0}-t_{{b_1}})\Delta}\bar
    P_{x_{b_1}}e^{-(t_{b_1}-t_{b_2})\Delta}\\
    \hspace{2.5in}\dotsm\bar P_{x_{{b_{m-1}}}}e^{-(t_{b_{m-1}}-t_{b_m})\Delta} \bar
    P_{x_{b_m}}e^{-(t_{b_m}-t_{1})\Delta}B_0.
  \end{multline}
  Noting that $\sum_{k=m}^{n-i}\binom{n-i-m}{k-m}(-1)^{k+m}=\uno{m=n-i}$ and recalling
  \eqref{eq:inv} we deduce that
  \begin{equation}
    \begin{aligned}
      &\left[(\sI+\sWo)^{-1}\sV\sP(\sT^-+\sT^+)(\sI+\sT^+)^{-1}\sK^0\right]_{i,j}
      =\uno{j=1}\sum_{i=b_0<b_1<\dots<b_{n-i}\leq n}V_{b_0}
      P_{x_{b_0}}e^{-(t_{b_0}-t_{{b_1}})\Delta}\\
      &\hspace{1.4in}\cdot\bar P_{x_{b_1}}e^{-(t_{b_1}-t_{b_2})\Delta}\dotsm\bar
      P_{x_{{b_{n-i-1}}}}e^{-(t_{b_{n-i-1}}-t_{b_{n-i}})\Delta} \bar
      P_{x_{b_{n-i}}}e^{-(t_{b_{n-i}}-t_{1})\Delta}B_0\\
      &\hspace{0.25in}=\uno{i=n,j=1}V_nP_{x_n}e^{-(t_n-t_1)\Delta}B_0\\
      &\hspace{0.4in}+\uno{i<n,j=1}V_iP_{x_i}e^{-(t_{i}-t_{{i+1}})\Delta}\bar
      P_{x_{i+1}}e^{-(t_{i+1}-t_{i+2})\Delta} \dotsm\bar
      P_{x_{n-1}}e^{-(t_{n-1}-t_{n})\Delta} \bar P_{x_n}e^{-(t_n-t_{1})\Delta}B_0.
    \end{aligned}\label{eq:preInduction}
  \end{equation}
  Post-multiplying by $(\sI+\sT^+)\sP\sV^{-1}$ we finally obtain from this and
  \eqref{eq:decsW} that
  \begin{multline}
    \label{eq:finaldecsW}
    \left[(\sI+\sWo)^{-1}\sWt\right]_{i,j}=\uno{i=n}V_nP_{x_n}e^{-(t_n-t_j)\Delta}B_0P_{x_j}V_j^{-1}\\
    +\uno{i<n}V_iP_{x_i}e^{-(t_{i}-t_{{i+1}})\Delta}\bar
    P_{x_{i+1}}e^{-(t_{i+1}-t_{i+2})\Delta} \dotsm\bar P_{x_n}e^{-(t_n-t_j)\Delta}B_0P_{x_j}V_j^{-1},
  \end{multline}
  where we have used again the fact that $e^{t\Delta}$ commutes with $B_0$ for $t>0$.

  At this stage we can check that $(\sI+\sWo)^{-1}\sWt$ is trace class. In fact it is
  enough to check (see (A.5) in \cite{bfp}) that each entry of this operator-valued matrix
  is trace class. The case $i=n$ was checked in Lemma A.3 in \cite{bfp}, while for the
  case $i<n$ we can use a similar strategy. Since $V_i$ and $V_j^{-1}$ are multiplication operators, they
  commute with $P_a$ for any $a$, and then choosing $-L\leq\min\{x_i,x_j\}$ we have
  \begin{multline}
    \left\|\left[(\sI+\sWo)^{-1}\sWt\right]_{i,j}\right\|_1
    =\left\|V_iP_{x_i}P_{-L}R_ie^{-(t_n-t_j)\Delta}B_0P_{-L}P_{x_j}V_j^{-1}\right\|_1\\
    =\left\|P_{x_i}P_{-L}V_iR_ie^{-(t_n-t_j)\Delta}B_0V_j^{-1}P_{-L}P_{x_j}\right\|_1
    \leq\left\|P_{-L}V_iR_ie^{-(t_n-t_j)\Delta}B_0V_j^{-1}P_{-L}\right\|_1,
  \end{multline}
  where $R_i=e^{-(t_{i}-t_{{i+1}})\Delta}\bar P_{x_{i+1}}e^{-(t_{i+1}-t_{i+2})\Delta}
  \dotsm\bar P_{x_n}$ and we have used the first of the inequalities
  \begin{equation}
    \|AB\|_1\leq\|A\|_{\rm op}\|B\|_1,\qquad\|AB\|_2\leq\|A\|_{\rm op}\|B\|_2,
    \qquad\|AB\|_1\leq\|A\|_1\|B\|_1,\label{eq:bdopnorm}
  \end{equation}
  with $\|\cdot\|_{\rm op}$ denoting the operator norm (see \cite{simon}) and
  $\|P_x\|_{\rm op}=1$. Next we remove the projections $P_{-L}$ and think instead of the
  operator $V_iR_ie^{-(t_n-t_j)\Delta}B_0V_j^{-1}$ as acting on $L^2([-L,\infty))$. Using
  again \eqref{eq:bdopnorm} and the fact that the operators $V_i$ and $V_i^{-1}$ commute
  with $\bar P_a$ we have that $\|V_iR_iV_n^{-1}\|_1$ is bounded by
  \[\|V_ie^{-(t_{i}-t_{i+1})\Delta}V_{i+1}^{-1}\bar
  P_{x_{i+1}}\|_1\|V_{i+1}e^{-(t_{i+1}-t_{i+2})\Delta}V_{i+2}^{-1}\bar
  P_{x_{i+2}}\|_1\dots\|V_{n-1}e^{-(t_{n-1}-t_n)\Delta}V_n^{-1}\|_1,\] which is finite
  because each factor is so by Lemma A.2 in \cite{bfp}. Since
  $\|V_ne^{-(t_n-t_j)\Delta}B_0V_j^{-1}\|_1$ (computed in $L^2([-L,\infty)$) is finite by
  Lemma A.3 in \cite{bfp} we deduce by \eqref{eq:bdopnorm} that
  \[\left\|\left[(\sI+\sWo)^{-1}\sWt\right]_{i,j}\right\|_1
  \leq\|V_iR_iV_n^{-1}\|_1\|V_ne^{-(t_n-t_j)\Delta}B_0V_j^{-1}\|_1<\infty.\]

  Going back to \eqref{eq:detDec}, since both $\sWo$ and $(\sI+\sWo)^{-1}\sWt$ are trace
  class, we have
  \begin{equation}
    \det\!\big(\sI-\sV\sK\sV^{-1}\big)=\det\!\big(\sI+\sWo\big)\det\!\big(\sI-(\sI+\sWo)^{-1}\sWt\big)
    =\det\!\big(\sI-(\sI+\sWo)^{-1}\sWt\big),\label{eq:detRed}
  \end{equation}
  where the second equality follows from the fact that, since $\sWo$ is strictly upper
  triangular, its only eigenvalue is 0, and thus $\det(\sI+\sWo)=1$. Now let $\sU$ be
  given by $\sU_{i,j}=U\uno{i=j}$ where $U$ is the (diagonal) multiplication operator
  introduced in Proposition \ref{prop:traceclass}. Then x to \eqref{eq:decsW} we
  have
  \[(\sI+\sWo)^{-1}\sWt=\tsWo\tsWt\]
  with $\tsWo=(\sI+\sWo)^{-1}\sV\sP(\sT^-+\sT^+)(\sI+\sT^+)^{-1}\sK^0\sU^{-1}$ and
  $\tsWt=\sU(\sI+\sT^+)\sP\sV^{-1}$.
  We have already checked that $\tsWo\tsWt$ is trace class, so if we prove that
  $\tsWt\tsWo$ is also trace class we can deduce from the cyclic property of determinants
  and \eqref{eq:detRed} that
  \begin{equation}
    \label{eq:detFinalExt}
    \det\!\big(\sI-\sV\sK\sV^{-1}\big)=\det\!\big(\sI-\tsWt\tsWo\big).
  \end{equation}

  Recall from \eqref{eq:preInduction} that only the first column of
  $(\sI+\sWo)^{-1}\sV\sP(\sT^-+\sT^+)(\sI+\sT^+)^{-1}\sK^0$ has non-zero entries. Since
  $\sU(\sI+\sT^+)\sP\sV^{-1}$ is upper triangular and $\sU^{-1}$ is diagonal, the same is true for
  $\tsWt\tsWo$. %Using \eqref{eq:inv} it is straightforward to check that all the $\sV$'s
  % and $\sV^{-1}$'s cancel from $\tsWt\tsWo$:
  % \[\tsWt\tsWo=\sU(\sI+\sT^+)\sP(\sI+\sWo)^{-1}(\sT^-+\sT^+)(\sI+\sT^+)^{-1}\sK^0\sU^{-1}.\]
  Observe that $\sV^{-1}(\sI+\sWo)^{-1}\sV=(\sI+\sP\sT^+\sP)^{-1}$, so all the $\sV$'s cancel
  in $\tsWt\tsWo$. For the first column of this operator-valued matirx we get using \eqref{eq:preInduction} that
  \begin{equation}
    \begin{aligned}
      &\big(\tsWt\tsWo\big)_{k,1}=Ue^{-(t_k-t_n)\Delta}P_{x_n}e^{-(t_n-t_1)\Delta}B_0U^{-1}\\
      &\qquad+\sum_{i=k}^{n-1}Ue^{-(t_k-t_i)\Delta}P_{x_{i}}e^{-(t_{i}-t_{i+1})\Delta}\bar
      P_{x_{i+1}}
      \dotsm\bar P_{x_{n-1}}e^{-(t_{n-1}-t_{n})\Delta}\bar P_{x_{n}}e^{-(t_n-t_1)\Delta}B_0U^{-1}\\
      &\quad=Ue^{-(t_k-t_n)\Delta}P_{x_n}e^{-(t_n-t_1)\Delta}B_0U^{-1}\\
      &\quad\qquad+\sum_{i=k}^{n-1}Ue^{-(t_k-t_i)\Delta}(I-\bar
      P_{x_{i}})e^{-(t_{i}-t_{i+1})\Delta}\bar P_{x_{i+1}}
      \dotsm\bar P_{x_{n-1}}e^{-(t_{n-1}-t_{n})\Delta}\bar P_{x_{n}}e^{-(t_n-t_1)\Delta}B_0U^{-1}\\
      &\quad=Ue^{-(t_k-t_n)\Delta}P_{x_n}e^{-(t_n-t_1)\Delta}B_0U^{-1}\\
      &\quad\qquad+\sum_{i=k}^{n-1}\bigg[Ue^{-(t_k-t_{i+1})\Delta}\bar P_{x_{i+1}} \dotsm\bar
      P_{x_{n-1}}e^{-(t_{n-1}-t_{n})\Delta}\bar P_{x_{n}}e^{-(t_n-t_1)\Delta}B_0U^{-1}\\
      &\hspace{1.4in}-Ue^{-(t_k-t_{i})\Delta}\bar P_{x_{i}} \dotsm\bar
      P_{x_{n-1}}e^{-(t_{n-1}-t_{n})\Delta}\bar P_{x_{n}}e^{-(t_n-t_1)\Delta}B_0U^{-1}\bigg].
    \end{aligned}
  \end{equation}
  Telescoping the last sum yields
  \begin{equation}
    \begin{aligned}
      \big(\tsWt\tsWo\big)_{k,1}&=Ue^{-(t_k-t_1)\Delta}B_0U^{-1} -U\bar
      P_{x_k}e^{-(t_k-t_{k+1})\Delta}\bar P_{x_{k+1}} \dotsm\bar
      P_{x_{n}}e^{-(t_n-t_1)\Delta}B_0U^{-1}\\
      &=U\!\left[e^{-(t_k-t_n)\Delta}-\bar P_{x_k}e^{-(t_k-t_{k+1})\Delta}\bar P_{x_{k+1}}
        \dotsm\bar P_{x_{n}}\right]\!e^{-(t_n-t_1)\Delta}B_0U^{-1}.
    \end{aligned}\label{eq:opFinal}
  \end{equation}
  Using this last decomposition we get directly from the proof of Proposition
  \ref{prop:traceclass}(a) that $\big(\tsWt\tsWo\big)_{k,1}$ is trace class. This justifies the identity
  \eqref{eq:detFinalExt}, and then since only the first column of $\tsWt\tsWo$ is non-zero
  we deduce that
  \noeqref{eq:opFinal}
  \[\det\!\big(\sI-\sV\sK\sV^{-1}\big)=\det\!\Big(I-\big(\tsWt\tsWo\big)_{1,1}\Big)_{L^2(\rr)}.\]
  The result now follows from the above formula for $\big(\tsWt\tsWo\big)_{k,1}$ with $k=1$.
\end{proof}

\begin{rem}\label{rem:airy2case2}
  A complete proof for the Airy$_2$ case can be obtained from the above argument by
  replacing $-\Delta$ by $H$, $B_0$ by $\K$, and both $\sV$ and $\sU$ by $\sI$. As we
  mentioned in Remark \ref{rem:airy2case1}, this case presents the additional issue that
  the operators $e^{tH}$ involved in $\sT^+$ and $\sT^-$ do not even map $L^2(\rr)$ to
  itself (in fact, note that $H$ has the whole real line as its spectrum). $\sT^-$, which
  is associated to operators $e^{tH}$ with $t>0$, presents no difficulty in the above
  proof. In fact, it always appears applied after $\sK$, which in this case is the
  diagonal matrix with $\K$ in each diagonal entry, so that since $\K$ projects onto the
  negative eigenspace of $H$ (see Remark \ref{rem:t}), each entry in $\sT^-\sK$ is a
  bounded operator acting on $L^2(\rr)$. This is analogous to the fact that, in the
  Airy$_1$ case, the operators $e^{-t\Delta}$ for $t>0$ always appear after $B_0$.

  To deal with $\sT^+$ we start with the formula
  \begin{equation}
    e^{-tH}f(x)=\int_{-\infty}^\infty dy\int_{-\infty}^\infty d\lambda\,e^{\lambda
      t}\Ai(x+\lambda)\Ai(y+\lambda)f(y).\label{eq:etH}
  \end{equation}
  One can check that for any $f\in L^2(\rr)$ the integral is convergent, and thus
  $e^{-tH}f$ is well defined, though not necessarily in $L^2(\rr)$. The key is to notice,
  again using the formula, that for any $a$ the operators $P_ae^{-tH}$ and $e^{-tH}P_a$
  are Hilbert-Schmidt (see \eqref{eq:PaetH}), so that
  $P_ae^{-tH}P_a=(P_ae^{-\frac{t}{2}H})(e^{-\frac{t}{2}H}P_a)$ is trace class by
  \eqref{eq:tracenormbd}. In particular, this implies that the operator $\sWo$ defined in
  \eqref{eq:sWs} (with $\sV=\sI$) is trace class in the Airy$_2$ case. To make sense of
  $(\sI+\sWo)^{-1}\sWt$, as needed in \eqref{eq:detDec}, we can use \eqref{eq:termInv}
  directly together with \eqref{eq:inv} to write
  \begin{multline}\label{eq:altForm}
    \left[(\sI+\sWo)^{-1}\sWt\right]_{i,j}
    =\sum_{k=0}^{n-i}(-1)^{k}\quad\smashoperator{\sum_{i<a_1<\dots<a_k\leq n}}\quad
    P_{x_{i}}e^{-(t_i-t_{a_1})H}P_{x_{a_1}}e^{-(t_{a_{1}}-t_{a_2})H}\\
    \dotsm P_{x_{a_{k-1}}}e^{-(t_{a_{k-1}}-t_{a_k})H}P_{x_{a_k}}e^{-(t_{a_k}-t_{j})H}\K P_{x_j}
  \end{multline}
  (cf. \eqref{eq:finaldecsW}), where the same argument can be applied to show
  that each term is well defined and is in fact trace class. This allows to derive
  \eqref{eq:detRed}, and it is easy to check that deriving \eqref{eq:detFinalExt} via the
  cyclic property of determinants involves no new difficulties.
  \noeqref{eq:altForm}

  A final remark is in order. The operator $\bar P_{x_1}e^{-(t_1-t_2)H}\dotsm
  e^{-(t_{n-1}-t_n)H}\bar P_{x_n}$ appearing in \eqref{eq:airy2fd} is ill-defined
  because, unlike in the preceding discussion, an operator of the form $\bar P_a
  e^{-tH}\bar P_b$ does not map $L^2(\rr)$ to itself. Hence \eqref{eq:airy2fd} should be
  understood as a shorthand notation for
  \begin{align}
    &\pp\!\left(\aip(t_1)\leq x_1,\dotsc,\aip(t_n)\leq x_n\right)\\
    &\hspace{0.3in}=\det\!\Bigg(I-\sum_{i=1}^n\sum_{k=0}^{n-i}(-1)^{k}\quad\smashoperator{\sum_{i<a_1<\dots<a_k\leq n}}\quad
    e^{-(t_1-t_i)H}P_{x_{i}}e^{-(t_i-t_{a_1})H}P_{x_{a_1}}e^{-(t_{a_{1}}-t_{a_2})H}\\
    &\hspace{2.4in}\dotsm P_{x_{a_{k-1}}}e^{-(t_{a_{k-1}}-t_{a_k})H}P_{x_{a_k}}e^{-(t_{a_k}-t_{1})H}\K\Bigg)_{L^2(\rr)},
  \end{align}
  which is obtained from the above proof by working directly with \eqref{eq:termInv}
  instead of \eqref{eq:preInduction}. Alternatively, one can rewrite
  \begin{multline}
    \pp\!\left(\aip(t_1)\leq x_1,\dotsc,\aip(t_n)\leq x_n\right)\\
    =\det\!\left(I-\left[e^{(t_1-t_n)H}-\bar P_{x_1}e^{(t_1-t_2)H}\bar
        P_{x_2}e^{(t_2-t_3)H}\dotsm\bar
        P_{x_n}\right]\!e^{(t_n-t_1)H}K_{\Ai}\right)_{L^2(\rr)}.
  \end{multline}
  The product inside this last determinant was shown to be trace class in Proposition
  \gref{prop:theta} (cf. Proposition \ref{prop:traceclass} below).
\end{rem}

Going back to the Airy$_1$ process, we turn next to proving the existence of trace class
operators which are conjugate to the ones appearing in \eqref{eq:airy1fd} and
\eqref{eq:basic1}. Given $\mathbf{x}=(x_1,\dots,x_n)$ and $\mathbf{t}=(t_1,\dots,t_n)$ with $t_i<t_{i+1}$
let
\begin{equation}\label{lambdathing}\Lambda^{\mathbf{x}}_{\mathbf{t}}=\bar P_{x_1}e^{-(t_1-t_2)\Delta}\bar P_{x_2}e^{-(t_2-t_3)\Delta}\dotsm
      e^{-(t_{n-1}-t_n)\Delta}\bar P_{x_n}.\end{equation}
For the case $t_i=\ell+\frac{i-1}{n-1}(r-\ell)$, $i=1,\dotsc,n$, and $x_i=g(t_i)$ for some $g\in
H^1([\ell,r])$ we write
\[\Lambda^g_{n,[\ell,r]}=\bar P_{g(t_1)}e^{-(t_1-t_2)\Delta}\bar P_{g(t_2)}e^{-(t_2-t_3)\Delta}\dotsm
  e^{-(t_{n-1}-t_n)\Delta}\bar P_{g(t_n)}.\]

Let $U$ be the operator defined by $Uf(x)=e^{-2(r-\ell)x}f(x)$.

\begin{prop}\label{prop:traceclass}
  Fix $\ell<r$ and let $g\in H^1([\ell,r])$.
  \begin{enumerate}[label=(\alph*)]
  \item $U\big(B_0-\Lambda^{\bf x}_{\bf t}e^{-(t_n-t_1)\Delta}B_0\big)U^{-1}$ and
    $U\big(B_0-\Lambda^{g}_{[\ell,r]}e^{-(r-\ell)\Delta}B_0\big)U^{-1}$ are trace class operators on
    $L^2(\rr)$.
  \item $\big\|U\big(B_0-\Lambda^{g}_{n,[\ell,r]}e^{-(r-\ell)\Delta}B_0\big)U^{-1}\big\|_1$
    is bounded uniformly in $n$.
  \item Let $n_k=2^k$. Then
    \[\lim_{k\to\infty}\big\|U\big(B_0-\Lambda^{g}_{n_k,[\ell,r]}e^{-(r-\ell)\Delta}B_0\big)U^{-1}-
      U\big(B_0-\Lambda^{g}_{[\ell,r]}e^{-(r-\ell)\Delta}B_0\big)U^{-1}\big\|_1=0.\]
  \end{enumerate}
\end{prop}

\begin{proof}
  The proof is similar to that of Proposition \gref{prop:theta}, although here using
  the conjugated kernels is crucial.

  Assume first that $g(t)=0$ and write $s=r-\ell$. We begin by considering the second
  operator in (a). Let $\varphi(z)=\sqrt{1+z^2}$ and write
  \[V(x,z)=\big(e^{s\Delta}-\Lambda^{g}_{[\ell,r]}\big)(x,z)e^{-2xs}\varphi(z)e^{-2zs}
  \quad\text{and}\quad
  W(z,y)=\big(e^{-s\Delta}B_0\big)(z,y)\varphi(z)^{-1}e^{2zs}e^{2ys}.\] Then
  \[U\Big(B_0-\Lambda^{g}_{[\ell,r]}e^{-s\Delta}B_0\Big)U^{-1}=VW.\] Since
  \begin{equation}
    \|VW\|_1\leq\|V\|_2\|W\|_2\label{eq:tracenormbd}
  \end{equation}
  (see \cite{simon}) it is enough to prove that
  $\|V\|_2<\infty$ and $\|W\|_2<\infty$.

  The estimate for $\|W\|_2$ is simple: using \eqref{eq:e-tDB0},
  \begin{equation}
  \begin{aligned}
    \|W\|^2_2&=\int_{\rr^2}dx\,dy\,
    \frac{e^{-4s^3/3+2(x+y)s}}{\varphi(x)^2}\!\Ai(x+y+s^2)^2
    =\int_{\rr^2}dx\,dy\,\frac {e^{-4s^3/3+2ys}}{\varphi(x)^2}\!\Ai(y+s^2)^2\\
    &=\|\varphi^{-1}\|^2_2\int_{-\infty}^\infty dy\,e^{-4s^3/3+2ys}\!\Ai(y+s^2)^2.
  \end{aligned}\label{eq:bdW}
  \end{equation}
  The last integral is finite thanks to the bounds
  \begin{equation}
    |\!\Ai(z)|\leq Ce^{-\frac23z^{3/2}}\text{ for }z\geq0,\qquad
    |\!\Ai(z)|\leq C\text{ for }z<0\label{eq:airyBd}
  \end{equation}
  for some constant $C>0$ (see (10.4.59-60) in \cite{abrSteg}), and thus $\|W\|_2<\infty$.

  For $V$, recalling that we are taking $g(t)=0$, we may shift time by $-(\ell+r)/2$ in
  the definition of $\Lambda^g_{[\ell,r]}$ to deduce that
  $\Lambda^g_{[\ell,r]}=\Lambda^g_{[-s/2,s/2]}$, and then by \eqref{eq:Lambda} we have
  \[\Lambda^{g}_{[\ell,r]}(x,y)
  =\frac{e^{-(x-y)^2/4s}}{\sqrt{4\pi s}}\pp_{\hat b(-s/2)=x,\hat b(s/2)=y} \!\left(\hat
    b(t)\leq0\,\text{ on }[-s/2,s/2]\right).\] Therefore
  \begin{equation}
    V(x,y)=\varphi(y)\frac{e^{-(x-y)^2/4s-2(x+y)s}}{\sqrt{4\pi s}}\pp_{\hat b(-s/2)=x,\hat b(s/2)=y}
    \!\left(\hat b(t)\geq 0\,\text{ for some }t\in[-s/2,s/2]\right).\label{eq:bridgeProb}
  \end{equation}
  The last crossing probability equals $e^{-xy/s}$ if $x\leq0,y\leq0$ and 1 otherwise (see
  page 67 in \cite{handbookBM}), and thus
  \begin{multline}
    \|V\|_2^2=\frac{1}{4\pi s}\int_{\rr^2\setminus(-\infty,0]^2}dx\,dy
    \,(1+y^2)\big[e^{-(x-y)^2/4s-2(x+y)s}\big]^2\\
    +\frac{1}{4\pi s}
    \int_{(-\infty,0]^2}dx\,dy\,(1+y^2)\big[e^{-(x+y)^2/4s-2(x+y)s}\big]^2.\label{eq:bdHS}
  \end{multline}
  Both Gaussian integrals can be easily seen to be finite, so we have shown that
  $\|V\|_2<\infty$.

  For the discrete time kernel we can use the same argument. To simplify notation we will
  write the proof for the kernel of the form $\Lambda^g_{n,[\ell,r]}$ (with $g=0$), the same proof
  works for $\Lambda^{\bf x}_{\bf t}$. We decompose the kernel as
  \[U\big(B_0-\Lambda^{g}_{n,[\ell,r]}e^{-(r-\ell)\Delta}B_0\big)U^{-1}=V_{n}W,\]
  where
  \[V_{n}(x,y)=\varphi(y)\frac{e^{-(x-y)^2/4s-2(x+y)s}}{\sqrt{4\pi s}}\pp_{\hat
    b^n(-s/2)=x,\hat b^n(s/2)=y} \!\left(\hat b^n(s)\leq 0\text{ on }[-s/2,s/2]\right)\]
  and $\hat b^n$ is a discrete time random walk with Gaussian jumps with mean 0 and
  variance $s/n$, started at time $-s/2$ at $x$, conditioned to hit $y$ at time $s/2$, and
  jumping at times $t^n_i=-s/2+\frac{i-1}{n-1}s$, $i\geq1$ (in the case of a kernel
  $\Lambda^{\bf x}_{\bf t}$ this random walk is not time-homogeneous, but this does not introduce any
  issues below). We deduce that
  \begin{multline}\label{eq:diff2}
    \big(e^{-(r-\ell)H}-\Lambda^g_{n,[\ell,r]}\big)(x,y)=\frac{\varphi(y)}{\sqrt{4\pi s}}
    e^{-(x-y)^2/4s-2(x+y)s}\\
    \cdot\pp_{\hat b^n(-s/2)=x,\hat b^n(s/2)=y}\!\left(\hat b^n(t^n_i)\geq 0\text{ for some
        }i\in\{1,\dotsc,n\}\right).
  \end{multline}
  A simple coupling argument (see the next paragraph) shows that the last probability is
  less than the corresponding one for the Brownian bridge, and thus we obtain for
  $\|e^{-(r-\ell)H}-\Lambda^g_{n,[\ell,r]}\|_2$ the same bound as the one we get for
  $\|e^{-(r-\ell)H}-\Lambda^g_{[\ell,r]}\|_2$ from \eqref{eq:bdHS}. This bound is, in
  particular, independent of $n$, so we have proved (a) and (b).

  To prove (c) we use again the above decompositions into $VW$ and $V_nW$. Our goal is to
  show that $\|V_{n_k}W-VW\|_1\to0$ as $k\to\infty$. Since
  $\|V_{n_k}W-VW\|_1\leq\|V_{n_k}-V\|_2\|W\|_2$ by \eqref{eq:tracenormbd} and we already
  know that $\|W\|_2<\infty$, all that is left is to show that
  \[\|V_{n_k}-V\|_2\xrightarrow[k\to\infty]{}0.\]
  Couple the Brownian bridge $\hat b$ and the conditioned random walk $\hat b^{n_k}$ by
  simply letting $\hat b^{n_k}(t^{n_k}_i)=\hat b(t^{n_k}_i)$ for each
  $i=1,\dotsc,{n_k}$. Since the Brownian bridge hits the positive half-line whenever the
  conditioned random walk does, it is clear that
  \begin{equation}
    \left|V_{n_k}(x,y)-V(x,y)\right|
    =\frac{e^{-(x-y)^2/4s-2(x+y)s}}{\sqrt{4\pi s}}q_{n_k}(x,y),\label{eq:VnkV}
  \end{equation}
  where $q_{n_k}(x,y)$
  is the probability that the Brownian bridge $\hat b(t)$ hits the positive half-line for $t\in[-s/2,s/2]$
  but not for any $t\in\{t^{n_k}_1,\dotsc,t^{n_k}_{{n_k}}\}$.  Since every point is
  regular for one-dimensional Brownian motion, $q_{n_k}(x,y)\searrow0$ as $k\to\infty$ for
  every fixed $x,y$, and thus the monotone convergence theorem yields \eqref{eq:VnkV}.

  To extend the result to $g\in H^1([\ell,r])$ we note that everything in the above
  argument deals with properties of a Brownian motion $b(s)$ killed at the positive
  half-line. In the general case we will have by \eqref{eq:Lambda} a Brownian motion
  $b(s)$ killed at the boundary $g(s)$ or, equivalently, a process $\tilde b(s)=b(s)-g(s)$
  killed at the positive half-line. Using the Cameron-Martin-Girsanov Theorem we can
  rewrite the probabilities for $\tilde b(s)$ in terms of probabilities for $b(s)$. Since
  $g(s)$ is a deterministic function in $H^1([\ell,r])$, the Radon-Nikodym derivative of
  $\tilde b(s)$ with respect to $b(s)$ has finite second moment, and thus by using the
  Cauchy-Schwarz inequality we get (a) and (b) from the above arguments. The convergence
  in (c) follows as well from the above arguments because it only depends on almost sure
  properties of the corresponding Brownian motion.
\end{proof}

\section{Regularity and continuum statistics}
\label{sec:cont}

We now use the Kolmogorov continuity criterion to prove the H\"older continuity of the
Airy$_1$ process (we will explain later how to adapt the proof to the Airy$_2$ case). An
important technical problem is that the kernel appearing inside the determinant in
\eqref{eq:airy1fd} is not trace class.

To apply the Kolmogorov criterion we have to get an appropriate bound on
\begin{equation} \det(I-B_0+\bar P_ae^{t\Delta}\bar P_b
  e^{-t\Delta}B_0)-\det(I-B_0+\bar P_a B_0).
\end{equation}
To deal with the fact that the kernels above are not trace class, we have to conjugate by
a kernel $U$ as in Proposition \ref{prop:traceclass}. The resulting bound in terms of
trace norms gets bad as $a,b\to -\infty$.  To get around this, we use the Kolmogorov
criterion in the following unusual form.

Given a  stochastic process
$X(t)$  and $M>0$ we denote by $X^{M}(t)$ the truncated process
\[X^{M}(t)=X(t)\uno{|X(t)|\leq M}+M\uno{X(t)>M}-M\uno{X(t)<-M}.\]

\begin{lem}\label{lem:kolmog}
  Let $X(t)$ be a real valued stochastic process  defined for $t$ in some
  interval $I\subseteq\rr$. Assume that the following two conditions hold:
  \begin{enumerate}[label=\arabic*.]
  \item There is a dense subset $J$ of $I$ such that $\lim_{K\to\infty}\pp(|X(t)|\leq
    K~\,\forall\,t\in J)=1$.
  \item There are $\alpha,\beta>0$ satisfying the following: for each $M>0$ there is an
    $\ep>0$ and $c>0$ such that
    \[\ee\!\left(|X^{M}(t)-X^{M}(s)|^\alpha\right)\leq c|t-s|^{1+\beta}\]
    for  all $s,t\in I$ with $|t-s|<\ep$.
  \end{enumerate}
  Then $X(t)$ has a version on $I$ with H\"older continuous paths with exponent $\frac\beta\alpha$.
\end{lem}

The lemma follows immediately from the usual Kolmogorov criterion, which, applied to 2,
shows that there is a version of $X(t)$ such that, for each $M>0$, $X^M(t)$ is H\"older
continuous with exponent $\frac\beta\alpha$. Such a function cannot be discontinuous if
it is bounded on a dense set.

In view of this lemma, after we verify the first condition (which we do in the next
result) it will be enough to consider the truncated process $\aipo^{M}(t)$. Throughout
this section all Fredholm determinants will be computed on $L^2(\rr)$, while $c$ and $c'$
will denote positive constants whose values may change from line to line.

\begin{lem}\label{lem:airy1bdd}
  Fix $L>0$ and write $D_L(n)=\{\tfrac{k}{2^{n+1}}L,\,k=-2^n,\dots,2^n\}$. Then
  \[\lim_{M\to\infty}\pp\!\left(\aipo(t)\leq M~\,\forall\,t\in\cup_{n>0}D_L(n)\right)=1.\]
\end{lem}

\begin{proof}
  By Theorem \ref{thm:airy1}, Proposition \ref{prop:traceclass}(c) and the bound
  \begin{equation}
    \label{eq:detBd}
      \big|\!\det(I+Q_1)-\det(I+Q_2)\big|\leq\|Q_1-Q_2\|_1e^{\|Q_1\|_1+\|Q_2\|_1+1}%\\&\leq\|Q_1-Q_2\|_1e^{\|Q_1-Q_2\|_1+2\|Q_2\|_1+1}
  \end{equation}
  for trace class operators $Q_1$ and $Q_2$ (see \cite{simon}), we have
  \begin{align}
    \pp\!\left(\aipo(t)\leq M~\,\forall\,t\in\cup_{n>0}D_L(n)\right)
    &=\lim_{n\to\infty}\pp\!\left(\aipo(t)\leq M~\,\forall\,t\in D_L(n)\right)\\
    &=\det\!\big(I-B_0+\Lambda^M_{[-L/2,L/2]}e^{-L\Delta}B_0\big),
  \end{align}
  where $\Lambda^M_{[-L/2,L/2]}$ denotes $\Lambda^g_{[-L/2,L/2]}$ with $g(t)=M$ and, we
  recall, the operator inside the determinant is trace class after conjugating by $U$ as
  in Proposition \ref{prop:traceclass}. Using \eqref{eq:detBd} again we deduce that it is
  enough to show that
  \begin{equation}
    \label{eq:Mnorm}
    \lim_{M\to\infty}\big\|U\big(B_0-\Lambda^M_{[-L/2,L/2]}e^{-L\Delta}B_0\big)U^{-1}\big\|_1=0.
  \end{equation}

  Following the proof of Proposition \ref{prop:traceclass}(a) we have
  \[\big\|U\big(B_0-\Lambda^M_{[-L/2,L/2]}e^{-L\Delta}B_0\big)U^{-1}\big\|_1\leq\|V\|_2\|W\|_2\]
  with $V$ and $W$ as in that proof. Recall that $W$ does not depend on $M$ and
  has finite Hilbert-Schmidt norm, so all we need is to show that $\|V\|_2\to0$. To estimate this last
  norm we can proceed exactly as in the arguments leading to \eqref{eq:bdHS}, only
  replacing $s$ by $L$ and the barrier at 0 for the Brownian bridge by a barrier at $M$,
  so that the corresponding crossing probability is now $e^{-(x-M)(y-M)/L}$ for $x,y\leq
  M$ and 1 otherwise. We obtain, after some simple manipulations,
  \begin{multline}
    \|V\|_2^2=\frac{1}{4\pi L}\int_{\rr^2\setminus(-\infty,M]^2}dx\,dy
    \,(1+y^2)\big[e^{-(x-y)^2/4L-2(x+y)L}\big]^2\\
    +\frac{1}{4\pi L}
    \int_{(-\infty,0]^2}dx\,dy\,(1+y^2)\big[e^{-(x+y)^2/4L-2(x+y)L-2ML}\big]^2.
  \end{multline}
  The last two integrals are easily seen to go to 0 as $M\to\infty$, and \eqref{eq:Mnorm} follows.
\end{proof}

Next we verify the second condition in Lemma \ref{lem:kolmog}. By the stationarity of $\aipo$ we may take $s=0$.

\begin{lem}\label{lem:truncated}
  Fix $\delta>0$. Then there is a $t_0\in(0,1)$ and $n_0\in\nn$ such that for $0<t<t_0$,
  $n\geq n_0$ and $M=\big(3\log(t^{-(1+n)})\big)^{1/3}$ we have
  \[\ee\!\left(\big[\aipo^{M}(t)-\aipo^{M}(0)\big]^{2n}\right)\leq ct^{1+(1-\delta)n}\]
  where the constant $c>0$ is independent of $\delta$, $n_0$ and $t_0$.
\end{lem}

\begin{proof}
  By the stationarity of the Airy$_1$ process
  \begin{equation}
    \ee\!\left(\big[\aipo^{M}(t)-\aipo^{M}(0)\big]^{2n}
      \uno{\aipo^{M}(0)\wedge\aipo^{M}(t)<-M}\right)
    \leq(2M)^{2n}\,2\hspace{0.05em}\pp(\aipo(0)<-M).
  \end{equation}
  Now $\pp(\aipo(0)<-M)=F_{\rm GOE}(-2M)\leq ce^{-\frac1{3}M^3}$ as $M\to\infty$ by the
  results of \cite{baikBuckDiF}. Hence we get
  \[\ee\!\left(\big[\aipo^{M}(t)-\aipo^{M}(0)\big]^{2n}
    \uno{\aipo^{M}(0)\wedge\aipo^{M}(t)<
      -M}\right)\leq c(2M)^{2n}t^{1+n}\leq ct^{1+(1-\delta)n}\] if $t$ is small enough. Thus it
  will be enough to prove the estimate
  \begin{equation}
    q(t):=\ee\!\left(\big[\aipo^{M}(t)-\aipo^{M}(0)\big]^{2n}
      \uno{\aipo^{M}(0)\wedge\aipo^{M}(t)\geq-M}\right)
    \leq ct^{1+(1-\delta)n}\label{eq:est}
  \end{equation}
  for small enough $t$.
  
  Let $F(a,b)=\pp(\aipo(0)\leq a,\aipo(t)\leq b)$ and $G(a)=\pp(\aipo(0)\leq b)$. Since
  $\frac{\p^2}{\p a\p b}G(a\wedge b)=0$ except when $a=b$ we have
  \[q(t)=\int_{-M}^\infty da\int_{-M}^\infty db\,(a-b)^{2n}\frac{\p^2}{\p a\p b}[F(a,b)-G(a\wedge b)].\]
  Truncating the upper limits at $K>0$ for a moment and integrating by parts the integral becomes
   \begin{align}
    &\int_{-M}^K da\left((a-K)^{2n}\frac\p{\p a}[F(a,K)-G(a)]-(a+M)^{2n}\frac\p{\p a}[F(a,-M)-G(-M)]\right)\\
    &\qquad\quad+\int_{-M}^K da\int_{-M}^K db\,2n(a-b)^{2n-1}\frac{\p}{\p a}[F(a,b)-G(a\wedge b)]\\
    &\quad=-2\int_{-M}^K da\left(2n(a-K)^{2n-1}[F(a,K)-G(a)]-2n(a+M)^{2n-1}[F(a,-M)-G(-M)]\right)\\
    &\qquad\quad-\int_{-M}^K da\int_{-M}^K db\,2n(2n-1)(a-b)^{2(n-1)}[F(a,b)-G(a\wedge b)]
  \end{align}
  (note that we have cancelled some boundary terms). We will see below in
  \eqref{eq:diffProbs2} that \[|F(a,K)-G(a)|\leq
  cM^{3/2}e^{1+cM^{3/2}}\int_{t^{-1/2}(K-a)}^\infty dx\,e^{-x^2/4},\]
  whence it is easy to see that the first integral on the right side above vanishes as
  $K\to\infty$. We deduce then that
  \begin{multline}
    \label{eq:qt}
    q(t)=4n\int_{-M}^\infty da\,(a+M)^{2n-1}[G(-M)-F(a,-M)]\\
    +2n(2n-1)\int_{-M}^\infty da\int_{-M}^\infty db\,(a-b)^{2(n-1)}[G(a\wedge b)-F(a,b)].
  \end{multline}

  We will estimate the last double integral, the first integral in the last line can be
  estimated similarly. Since the integrand is symmetric, it will be enough to restrict
  the integral to the case $-M\leq a\leq b$. Using the definitions of $F$ and $G$ and
  Theorem \ref{thm:airy1} we have
  \begin{equation}
    F(a,b)-G(a\wedge b)=\det(I-B_0+\bar P_ae^{t\Delta}\bar P_b
    e^{-t\Delta}B_0)-\det(I-B_0+\bar P_a B_0).\label{eq:diffProbs}
  \end{equation}
  Recall that the operator inside the first
  determinant is trace class after conjugating by the kernel $U$ introduced in Proposition
  \ref{prop:traceclass}. We will use the bound
  \begin{equation}
    \big|\!\det(I+Q_1)-\det(I+Q_2)\big|\leq\|Q_1-Q_2\|_1e^{\|Q_1-Q_2\|_1+2\|Q_2\|_1+1},\label{eq:detBd2}
  \end{equation}
  which follows directly from \eqref{eq:detBd}, to estimate the difference of determinants in
  \eqref{eq:diffProbs}, so our first task will be to estimate the trace norms of the
  operators
  \begin{equation}
    Q_2-Q_1=U\big(\bar P_ae^{t\Delta}\bar P_be^{-t\Delta}B_0-\bar
    P_aB_0\big)U^{-1}\qquad\text{and}\qquad Q_1=U\big(\bar P_aB_0-B_0\big)U^{-1}\label{eq:A1A2}
  \end{equation}
  for $-M\leq a\leq b$.

  We will use a different approach, and in particular a different choice of the kernel $U$, than the
  one used in the proof of Proposition \ref{prop:traceclass}. In what follows we will
  write $\tilde x=2^{1/3}x$ and $\tilde y=2^{1/3}y$. Let
  \[Uf(x)=e^{-(t+\alpha)\tilde x}\phi(\tilde x),\quad\text{where}\quad\phi(x)=e^{-\alpha x}\uno{x\geq-2^{1/3}M}+\uno{x<-2^{1/3}M}\]
  and $\alpha=M^{-1}$. We bound first the norm of $Q_1$. Using the identity
  \begin{equation}
    \label{eq:airyConv}
    \int_{-\infty}^\infty du\Ai(a+u)\!\Ai(b-u)=2^{-1/3}\Ai(2^{-1/3}(a+b))
  \end{equation}
  we have
  \begin{equation}
  \begin{aligned}
    Q_1=-2^{1/3}Q_1^1Q_1^2\qquad\text{with}\quad
    Q_1^1(x,u)&=\uno{x\geq a}e^{-(t+\alpha)\tilde x}\phi(\tilde x)^{-1}\Ai(\tilde
    x+u)e^{(t+\alpha/2)u},\\
    Q_1^2(u,y)&=e^{(t+\alpha)\tilde y}\phi(\tilde y)\Ai(\tilde y-u)e^{-(t+\alpha/2)u}.
  \end{aligned}\label{eq:A11A12}
  \end{equation}
  Now (using the fact that $a\geq-M$)
  \[\|Q_1^1\|^2_2=\int_a^\infty dx\int_{-\infty}^\infty
    du\,e^{-2t\tilde x}\Ai(\tilde x+u)^2e^{(2t+\alpha)u}
    =\int_a^\infty dx\,e^{-(4t+\alpha)\tilde x}\int_{-\infty}^\infty
    du\Ai(u)^2e^{(2t+\alpha)u}.\]
  By \eqref{eq:airyBd} the last integral in $u$ is bounded by $c(t+\alpha)^{-1/2}$, and
  then
  \[\|Q_1^1\|_2\leq c(t+\alpha)^{-3/4}e^{-c(t+\alpha)a}\leq c'M^{3/4},\]
  where the second inequality follows from the choice $\alpha$ and $M$ and the fact that
  $a\geq-M$. For $Q_1^2$ we have
    \begin{align}
    \|Q_1^2\|^2_2&=\int_{-\infty}^\infty dy\int_{-\infty}^\infty
    du\,e^{2(t+\alpha)\tilde y}\phi(\tilde y)^2\Ai(\tilde y-u)^2e^{-(2t+\alpha)u}\\
    &=\int_{-\infty}^\infty dy\,e^{\alpha\tilde y}\phi(\tilde y)^2\int_{-\infty}^\infty
    du\Ai(-u)^2e^{-(2t+\alpha)u}.
  \end{align}
  The $u$ integral is bounded by $c(t+\alpha)^{-1/2}$ as before, while the $y$ integral
  equals
  \[\int_{-\infty}^{-M} dy\,e^{\alpha\tilde y}+\int_{-M}^{\infty} dy\,e^{-\alpha\tilde y}
  \leq c\alpha^{-1}e^{\alpha M}\]
  so we also have $\|Q_1^2\|_2\leq cM^{3/4}$. Using these two estimates with
  \eqref{eq:tracenormbd} and \eqref{eq:A11A12} we conclude that
  \begin{equation}
    \label{eq:bdA1}
    \|Q_1\|_1\leq cM^{3/2}.
  \end{equation}\noeqref{eq:bdA1}

  Now we need to bound $\|U(Q_2-Q_1)U^{-1}\|_1$. Recall that we are assuming $a\leq b$, so
  that $\bar P_a(e^{t\Delta}\bar P_b-\bar P_be^{t\Delta})=-\bar P_a e^{t\Delta}P_b$. Then
  \begin{align}
    U(Q_2-Q_1)U^{-1}(x,y)&=-\uno{x\leq a}e^{-(t+\alpha)\tilde x}\phi(\tilde x)^{-1} \int_{b}^\infty
    dz\,\frac1{\sqrt{4\pi t}}e^{-(x-z)^2/4t}\\
    &\hspace{1.7in}\cdot e^{-2t^3/3-(z+y)t}\Ai(z+y+t^2)e^{(t+\alpha)\tilde y}\phi(\tilde y)\\
    &=-\int_{-\infty}^\infty d\tilde z\,\frac1{\sqrt{4\pi}}e^{-\tilde z^2/4}\uno{\sqrt{t}\tilde z\geq b-x}
    \uno{x\leq a}e^{-(t+\alpha)\tilde x}\phi(\tilde x)^{-1}\\
    &\hspace{1in}\cdot e^{-2t^3/3-(x+y+\sqrt{t}\tilde z)t}\Ai(x+y+\sqrt{t}\tilde z+t^2)e^{(t+\alpha)\tilde y}\phi(\tilde y)
  \end{align}
  where we performed the change of variables $z=x+\sqrt{t}\tilde z$. We regard this as an
  average of the kernels $C_{\tilde z}(x,y)$ given by
  \[C_{\tilde z}(x,y)=\uno{\sqrt{t}\tilde z\geq b-x,\,x\leq a}
    \phi(\tilde x)^{-1}e^{-2t^3/3-(x+\tilde x)t-(y-\tilde y)t-\alpha(\tilde x-\tilde y)+t^{3/2}\tilde z}
    \!\Ai(x+y+\sqrt{t}\tilde z+t^2)\phi(\tilde y),\]
  so that
  \[\big\|U(Q_2-Q_1)U^{-1}\big\|_1\leq\int_{-\infty}^\infty d\tilde
  z\,\frac1{\sqrt{4\pi}}e^{-\tilde z^2/4}\|C_{\tilde z}\|_1\leq\int_{\tfrac{b-a}{\sqrt{t}}}^\infty d\tilde
  z\,\frac1{\sqrt{4\pi}}e^{-\tilde z^2/4}\|C_{\tilde z}\|_1,\]
  where the second inequality follows from the fact that $C_{\tilde z}$ vanishes for
  $\sqrt{t}\tilde z<b-a$. The same argument as the one used to estimate $\|Q_1\|_1$ with
  only a bit of extra arithmetic gives the same bound for $\|C_{\tilde z}\|_1$ % a similar bound
  % \[\|C_{\tilde z}\|_1\leq cM^{3/2}\big[e^{-c'\alpha a}-e^{-c'\alpha(b-\sqrt{t}\tilde z)}\big]\]
  % (the additional exponential terms come from the fact that now $x$ is restricted to
  % $[b-\sqrt{t}\tilde z,a]$ instead of $[a,\infty)$, 
  and thus we get
  \[\big\|U(Q_2-Q_1)U^{-1}\big\|_1\leq cM^{3/2}\Phi(t^{-1/2}(b-a))\]
  with $\Phi(x)=\int_x^{\infty}dz\,e^{-z^2/4}$ (in fact a better bound can be obtained in
  this case without much difficulty, but we will not need it below).

  Using the bounds on $\big\|UQ_1U^{-1}\big\|_1$ and $\big\|U(Q_2-Q_1)U^{-1}\big\|_1$ in
  \eqref{eq:diffProbs} and \eqref{eq:detBd2} we deduce that
  \begin{equation}
    \big|F(a,b)-G(a\wedge b)\big|\leq cM^{3/2}\Phi(t^{-1/2}(b-a))e^{1+cM^{3/2}}\leq ct^{-1}\Phi(t^{-1/2}(b-a))\label{eq:diffProbs2}
  \end{equation}
  by our choice of $M$. Therefore
  \begin{align}
    \int_{-M}^\infty da\int_{-M}^\infty db\,(a&-b)^{2(n-1)}[G(a\wedge b)-F(a,b)]\\
    &\leq ct^{-1}\int_{-M}^\infty da\int_{-M}^\infty db\,(a-b)^{2(n-1)}\Phi(t^{-1/2}(b-a))\\
    &=ct^{n-3}\int_{-M}^\infty da\int_{-M}^\infty db\,(a-b)^{2(n-1)}\Phi(b-a).
    %\leq ct^{-c'(1+n)}...
  \end{align}
  Using the standard estimate $\Phi(x)\leq ce^{-x^2/4}$ as $x\to\infty$ it is not hard to
  see that the last integral is bounded by $cM^{2(n-1)}$. Using this in the second
  integral in \eqref{eq:qt}, and recalling that a similar estimate holds for the
  first integral, we deduce that
  \[q(t)\leq cn^2M^{2(n-1)}t^{n-3}\]
  and thus, using our choice of $M$, \eqref{eq:est} follows.
\end{proof}

\begin{proof}[Proof of Theorem \ref{thm:regularity}]
  The last two lemmas allow to check the hypotheses of Lemma \ref{lem:kolmog}, which
  yields the result for the Airy$_1$ case.

  The proof for the Airy$_2$ case is slightly simpler because the operators involved are
  trace class, and can be obtained by adapting the preceding arguments as we explain next.
   
  The one-point marginal of $\aip$, which is given by the Tracy-Widom GUE distribution,
  satisfies the tail estimate $F_{\rm GUE}(-M)\leq ce^{-\frac1{12}|M|^3}$ (see
  \cite{tracyWidom}). Choosing now $M=\big(12\log(t^{-(1+n)})\big)^{1/3}$ it is not hard
  to check that the main argument used in the case of the Airy$_1$ process works in
  exactly the same way if we change our determinantal formulas to the corresponding ones
  for $\aip$. Thus all we need to do is to obtain an analogous estimate on the difference
  \[F(a,b)-G(a\wedge b)=\det\!\big(I-\K+\bar P_ae^{-tH}\bar P_be^{tH}\K\big)-
  \det\!\big(I-\K+\bar P_a\K\big)\] for $-M\leq a\leq b$. Recall that the operators inside
  these determinants are trace class in this case, so there will be no need to conjugate.
  Proceeding as in the proof for $\aipo$ we need to bound the trace norms of the operators
  \[Q_2-Q_1=\bar P_ae^{-tH}\bar P_be^{tH}\K-\bar P_a\K\qquad\text{and}\qquad Q_1=\bar P_a
  \K-\K.\]

  We start with $Q_1$, which we rewrite as $-(P_ae^{-\alpha H}N)(N^{-1}e^{\alpha H}\K)$
  with $\alpha=M^{-1}$ and $N$ the multiplication operator $Nf(x)=\varphi(x)f(x)$ with
  $\varphi(x)=(1+x^2)^{1/2}$ (the choice of $\varphi$ is not particularly important). It
  is easy to check (see \geqref{eq:sndHS}) that
  \[\big\|N^{-1}e^{\alpha H}\K\big\|_2^2<c\hspace{0.05em}\alpha^{-1}\]
  for some $c>0$. On the other hand,
  \begin{equation}\label{eq:PaetH}
    \begin{split}
      \big\|P_ae^{-\alpha H}\big\|_2^2&=\int_{a}^\infty dx\int_{-\infty}^\infty
      dy\int_{\rr^2}d\lambda\,d\tilde\lambda
      \,e^{-\alpha(\lambda+\tilde\lambda)}\Ai(x-\lambda)\Ai(y-\lambda)\Ai(x-\tilde\lambda)\Ai(y-\tilde\lambda)\\
      &=\int_a^\infty dx\int_{-\infty}^\infty
      d\lambda\,e^{-2\alpha\lambda}\Ai(x-\lambda)^2
      =\int_a^\infty dx\,e^{-2\alpha x}\int_{-\infty}^\infty d\lambda\,e^{-2\alpha\lambda}\Ai(-\lambda)^2\\
      &\leq c\alpha^{-3/2}e^{-2\alpha a},
    \end{split}
  \end{equation}
  where we used \eqref{eq:airyBd} as before. Using these two bounds together with
  \eqref{eq:tracenormbd}, our choice of $\alpha$ and the fact that $a\geq-M$, we get
  \begin{equation}
    \|Q_1\|_1\leq c\alpha^{-5/4}e^{-\alpha a}\leq c'M^{5/4}.\label{eq:obd}
  \end{equation}

  We turn next to the trace norm of $Q_2-Q_1$. Recalling that $H=-\Delta+x$ and defining
  the multiplication operator $(e^{\alpha\xi}f)(x)=e^{\alpha x}f(x)$ (the reason we use
  the letter $\xi$ instead of $x$ in the definition is that we will use the operator at
  points other than $x$ below), one can derive formally, using the
  Baker-Campbell-Hausdorff formula, that
  \[e^{-tH}=e^{t\Delta}e^{t^3/3+t^2\nabla}e^{-t\xi},\] where $e^{t^2\nabla}f(x)=f(x+t^2)$
  (see \cite{qr-airy1to2} for a similar computation). This formula can then be checked
  directly by integration using \eqref{eq:etH} and therefore we may write, similarly to
  the Airy$_1$ case,
  \begin{align}
    (Q_2-Q_1)(x,y)&=\uno{x\leq a} \int_{-\infty}^\infty
    dz\,\frac{1}{\sqrt{4\pi t}}e^{-(x-z)^2/4t}\big(e^{t^3/3+t^2\nabla}e^{-t\xi}P_be^{tH}\K\big)(z,y)\\
    &=\uno{x\leq a} \int_{-\infty}^\infty d\tilde z\,\frac{1}{\sqrt{4\pi}}e^{-\tilde
      z^2/4}\big(e^{t^3/3+t^2\nabla}e^{-t\xi}P_be^{tH}\K\big)(\sqrt{t}\tilde z+x,y)\\
    &=\int_{\frac{b-a-t^2}{\sqrt{t}}}^\infty d\tilde z\,\frac{1}{\sqrt{4\pi}}e^{-\tilde
      z^2/4}C_{\tilde z}(x,y),
  \end{align}
  where $C_{\tilde z}=\bar P_ae^{t^3/3+(\sqrt{t}\tilde z+t^2)\nabla}e^{-t\xi}P_be^{tH}\K$
  and we have used the fact that $C_{\tilde z}$ vanishes for $\sqrt{t}\tilde
  z<b-a-t^2$. Proceeding as above we write, with $\alpha=M^{-1}$,
  \begin{align}
    \|C_{\tilde z}\|_1&\leq\|\bar P_ae^{t^3/3+(\sqrt{t}\tilde
      z+t^2)\nabla}e^{-t\xi}P_be^{(t-\alpha)H}\|_2 \|e^{\alpha H}\K\|_2\\
    &\leq\|\bar P_ae^{t^3/3+(\sqrt{t}\tilde z+t^2)\nabla}e^{-t\xi}P_b\|_{\rm
      op}\|P_be^{(t-\alpha)H}\|_2\|e^{\alpha H}\K\|_2,
  \end{align}
  where $\|\cdot\|_{\rm op}$ denotes the operator norm in $L^2(\rr)$ and we have used
  \eqref{eq:bdopnorm}. The first norm on the second line can be easily bounded by
  $ce^{-2t^3/3-tb-t^{3/2}\tilde z}$, while for the other two norms we have already
  obtained $\|P_be^{(t-\alpha)H}\|_2\leq c(\alpha-t)^{-3/4}e^{(t-\alpha)b}$ and
  $\|e^{\alpha H}\K\|_2=(2\alpha)^{-1/2}$ in the derivation of \eqref{eq:obd}. Since we
  are only interested in the case $\sqrt{t}\tilde z\geq b-a-t^2$, we have
  $e^{-2t^3/3-tb-t^{3/2}\tilde z}\leq e^{t^3/3-2tb+ta}$ and then
  \[\|C_{\tilde z}\|_1\leq c(\alpha-t)^{-3/4}\alpha^{-1/2}e^{t^3/3-(t+\alpha)b+ta}\leq
  c'M^{5/4},\] where we have used the again our choice of $M$ and $\alpha$ and the fact
  that $-M\leq a\leq b$. Plugging this in the above formula for $Q_2-Q_1$ we get
  \[\|Q_2-Q_1\|_1\leq cM^{5/4}\Phi(t^{-1/2}(b-a-t^2)).\] This estimate, together with the
  one for $\|Q_1\|_1$, allows to derive the an estimate analogous to
  \eqref{eq:diffProbs2}:
  \[\big|F(a,b)-G(a\wedge b)\big|\leq cM^{5/4}\Phi(t^{-1/2}(b-a-t^2))e^{1+cM^{5/4}}\leq
  ct^{-1}\Phi(t^{-1/2}(b-a-t^2)).\] Comparing with \eqref{eq:diffProbs2}, the only
  difference is the additional shift by $-t^{3/2}$ in the error function $\Phi$, but it is
  easy to see that this does not introduce any difficulty, and the rest of the proof
  follows as for $\aipo$.
\end{proof}

Finally we turn to the continuum statistics formula for the Airy$_1$ process.

\begin{proof}[Proof of Theorem \ref{thm:aiLo}]
  With the notation introduced before Proposition \ref{prop:traceclass} we have
  \[\pp\!\left(\aipo(t_1)\leq g(t_1),\dots,\aipo(t_{n_k})\leq g(t_{n_k})\right)
  =\det\!\left(I-B_0+\Lambda^g_{n_k,[\ell,r]}e^{-(r-\ell)\Delta}B_0\right)_{L^2(\rr)},\]
  where $n_k=2^k$. Since, by Theorem \ref{thm:regularity}, $\aipo$ has a continuous
  version, the probability on the left side converges to $\pp(\aipo(t)\leq g(t)~\,\forall
  t\in[\ell,r])$, and thus it is enough to show that
  \begin{multline}
    \lim_{k\to\infty}\det\!\Big(I-U\big(B_0-\Lambda^g_{n_k,[\ell,r]}e^{-(r-\ell)\Delta}B_0\big)U^{-1}\Big)_{L^2(\rr)}\\
    =\det\!\Big(I-U\big(B_0-\Lambda^g_{[\ell,r]}e^{-(r-\ell)\Delta}B_0\big)U^{-1}\Big)_{L^2(\rr)},
  \end{multline}
  where $n_k=2^k$. Since $A\mapsto\det(I+A)$ is a continuous function on the space of
  trace class operators by \eqref{eq:detBd}, the identity follows readily from Proposition
  \ref{prop:traceclass}(c).
\end{proof}

\section{Local Brownian property of Airy$_1$}

Note that, by stationarity and time reversibility, it is enough to study the finite
dimensional distribution of $\aipo$ at times $s=0<t_1<\cdots<t_n$. We have the following
formula for the Airy$_1$ process conditioned at a point.

\begin{lem}\label{lem:cond}
  For $0<t_1<\cdots< t_n$,
%\begin{eqnarray}
% && 
% \pp\!\left(\aipo( t_1)\leq x+y_1 , \ldots, \aipo( t_n)\leq x+y_n~\mid~\aipo(0)=x \right) \\
% &&  = -\frac{ \pp\!\left(\aipo(0)\le x , \aipo( t_1)\leq x+y_1 , \ldots, \aipo( t_n)\leq x+y_n\right) 
%\tr \left( \left(I-B_0 + \Lambda_{ (0,\mathbf{t}) }^{(x,\mathbf{y}+x)} e^{-t_n\Delta} B_0\right)^{-1}
%\delta_x e^{t_1\Delta}\Lambda_{ \mathbf{t} }^{\mathbf{y}+x} e^{-t_n\Delta} B_0 \right)}{F'_{\rm GOE} (x)}
% \end{eqnarray}
 \begin{multline}
 \pp\!\left(\aipo( t_1)\leq x+y_1 , \ldots, \aipo( t_n)\leq x+y_n\,\middle|\,\aipo(0)=x \right) \\
  = -\frac1{2F'_{\rm GOE} (2x)} \pp\!\left(\aipo(0)\le x , \aipo( t_1)\leq x+y_1 , \ldots,
    \aipo( t_n)\leq x+y_n\right) \\
  \cdot\tr\!\left[\left(I-B_0 + \Lambda_{ (0,\mathbf{t}) }^{(x,\mathbf{y}+x)} e^{-t_n\Delta} B_0\right)^{-1}
\delta_x e^{t_1\Delta}\Lambda_{ \mathbf{t} }^{\mathbf{y}+x} e^{-t_n\Delta} B_0 \right]
\label{ccd}
 \end{multline}
where $\Lambda_\mathbf{t}^\mathbf{x}$ is defined in \eqref{lambdathing} and $(0,\mathbf{t})$ 
 and $(x,\mathbf{y}+x)$ are notations for the vectors $(0,t_1,\ldots,t_n)$ and $(x,y_1+x, \ldots, y_n+x)$.
 \end{lem}

 Note again that the analogous formula is true for Airy$_2$. We remark that in the trace
 appearing in \eqref{ccd} we should be conjugating by the operator $U$ introduced before
 Proposition \ref{prop:traceclass} to make sure that the operator is trace class. The same
 is true for the calculations that follow. To simplify the argument we will ignore these
 conjugations and skip some details throughout this section, we hope that at this point
 the reader can fill in the necessary arguments.

 \begin{proof}[Proof of Lemma \ref{lem:cond}]
 Note first that
 \begin{align}
   &\pp\!\left(\aipo( t_1)\leq x+y_1 , \ldots, \aipo( t_n)\leq x+y_n\,\middle|\,\aipo(0)=x
   \right)\\
   &\hspace{0.5in}=\frac{1}{2F'_{\rm GOE}(2x)}\p_h\!\left.\pp\big(\aipo(0)\leq h,\,\aipo(
     t_1)\leq x+y_1 , \ldots, \aipo( t_n)\leq x+y_n\big)\right|_{h=x}\\
   &\hspace{0.5in}=\frac{1}{2F'_{\rm GOE}(2x)}\p_h\!\left.\det\!\left(I-B_0 + \Lambda_{
         (0,\mathbf{t}) }^{(x,\mathbf{y}+x)} e^{-t_n\Delta} B_0\right)\right|_{h=x},
 \end{align}
 where we have used the fact that $\pp(\aipo(0)\leq x)=F_{\rm GOE}(2x)$ and Theorem
 \ref{thm:airy1}. Now recall (see \cite{simon}) that if $\{A(\beta)\}_{\beta\geq0}$ is
 family of trace class operators which is Fr\'echet differentiable (in trace class norm)
 at $\beta=h$ then
 \begin{equation}
   \p_h\det(I+A(h))=\det(I+A(h))\tr[(I+A(h))^{-1}\p_hA(h)].\label{eq:derDet}
 \end{equation}
 The result now follows from
 computing the Fr\'echet derivative of $\Lambda_{ (0,\mathbf{t}) }^{(h,\mathbf{y}+x)}$,
 which can be shown without difficulty (after introducing the necessary conjugations) to make sense in
 trace class norm.
 \end{proof}

 \begin{proof}[Proof of Theorem \ref{thm:localBM}]
   We study the last line of \eqref{ccd} and to make it easier to read we call $L=B_0 +
   \Lambda_{ (0,\mathbf{t}) }^{(x,\mathbf{y}+x)} e^{- t_n\Delta} B_0$. Note first of all
   that it is given explicitly by
   \begin{multline}
     \tr\!\left[ \left(I-L\right)^{-1}
       \delta_x e^{ t_1\Delta}\Lambda_{ \mathbf{t} }^{\mathbf{y}+x} e^{-t_n\Delta} B_0 \right]\\
     = \int_{-\infty}^\infty dz\, e^{t_1\Delta }\bar{P}_{x+y_1} \cdots e^{(t_n-t_1)\Delta
     }\bar{P}_{x+y_n}(x,z) \int_{-\infty}^\infty du\, e^{-t_n\Delta } B_0(z,u)
     \left(I-L\right)^{-1}\!(u,x).
   \end{multline}
   Shifting $z$ by $x$ and using the translation invariance of the heat operators we can
   rewrite the trace as
   \[\int _{-\infty}^\infty dz\, e^{t_1\Delta }\bar{P}_{y_1} \cdots e^{(t_n-t_{n-1})\Delta
   }\bar{P}_{y_n}(0,z) \int _{-\infty}^\infty du\, e^{-t_n\Delta } B_0(z+x,u)
   \left(I-L\right)^{-1}\!(u,x).\] If we put in the Brownian scaling $\mathbf{t}\mapsto
   \ep \mathbf{t}$, $\mathbf{y}\mapsto \sqrt{\ep} \mathbf{y}$ we get
   \[\int_{-\infty}^\infty dz\, e^{\ep t_1\Delta }\bar{P}_{\sqrt{\ep}y_1} \cdots
   e^{\ep(t_n-t_1)\Delta }\bar{P}_{\sqrt{\ep}y_n}(0,z) \int _{-\infty}^\infty du\, e^{-\ep
     t_n\Delta } B_0(z+x,u) \left(I-L_\ep\right)^{-1}\!(u,x),\] where $L_\ep$ is defined
   in the obvious way by introducing the Brownian scaling in $L$. Since the heat operators
   are invariant under this scaling we can change $z\mapsto \sqrt{\ep} z$ to see that this
   is equal to
   \[\int_{-\infty}^\infty \, dz e^{t_1\Delta }\bar{P}_{y_1} \cdots e^{(t_n-t_{n-1})\Delta
   }\bar{P}_{y_n}(0,z) \int_{-\infty}^\infty \, du e^{-\ep t_n\Delta }
   B_0(\sqrt{\ep}z+x,u) \left(I-L_\ep\right)^{-1}\!(u,x).\] Combined with
   $\frac{d}{dx}F_{\rm GOE} (2x)= -{F_{\rm GOE} (2x)} \int_{-\infty}^\infty du
   \,B_0(x,u)\left(I-B_0+\bar{P}_x B_0\right)^{-1}\!(u,x) $, which follows easily from
   \eqref{eq:derDet}, we obtain \eqref{i} from this and \eqref{ccd}. Now \eqref{I} goes to 1 as
   $\ep\to0$ by the continuity of Airy$_1$ proved in Theorem \ref{thm:airy1}. On the other
   hand, one can show that $L_\ep$ converges to $L$ as $\ep\to0$ in trace class norm,
   which implies (see \cite{simon}) that $(I-L_\ep)^{-1}\to(I-L)^{-1}$ in the same
   sense. Using this it is not hard to show by the dominated convergence theorem that
   \eqref{II} goes to $1$ as $\ep\to0$. This implies the convergence of the finite
   dimensional distributions to those of Brownian motion, and thus concludes the proof.
 \end{proof}

\printbibliography[heading=apa]

\end{document}